\documentclass[10pt]{article}

\usepackage{latexsym,color,amsmath,amsthm,amssymb,amscd,amsfonts}

\newtheorem{theo}{Theorem}

\newtheorem{cor}[theo]{Corollary}
\newtheorem{prop}[theo]{Proposition}
\newtheorem{defn}[theo]{Definition}
\newtheorem{exam}[theo]{Example}
\def\R{\R}

\def\E{\mathcal{E}}

\def\D{{D}}

\def\R{{\mathbb R}}
\def\grad{\nabla}

\def\qed{\hfill $\vcenter{\hrule height .3mm
\hbox {\vrule width .3mm height 2.1mm \kern 2mm \vrule width .3mm
height 2.1mm} \hrule height .3mm}$ \bigskip}

\def\lam{\lambda}

\def\to{\rightarrow}

\def\pmx{\begin{pmatrix}}
\def\emx{\end{pmatrix}}

\def\det{{\rm det}}

\def\R{\mathbb R}

\def\D{\mathcal D}

\begin{document}

\title{
Pinsker inequalities and related Monge-Amp\`ere  equations for log concave functions 
}

\author{Umut Caglar, Alexander V. Kolesnikov\thanks{Supported by RFBR project
20-01-00432; The article was prepared within the framework of the HSE University Basic Research Program.} \ and Elisabeth M.  Werner\thanks{Partially supported by  NSF grant DMS-1811146}}
\date{}

\maketitle
\begin{abstract}

In this paper we further develop  the theory of $f$-divergences for log-concave  functions and their related inequalities. 
We establish Pinsker  inequalities and new affine invariant entropy inequalities.   We obtain new inequalities on  functional affine surface area and 
lower and upper bounds for the Kullback-Leibler divergence  in terms of functional affine surface area.  The functional inequalities 
lead to new inequalities for $L_p$-affine surface areas  for convex bodies.
\par
\noindent
Equality characterizations in these inequalities are related to a Monge Amp\`ere differential equation. We establish uniqueness of the solution of the equation.
\end{abstract}

\section{Introduction}

Information theory, probability theory, and  statistics have become important in convex geometry and vice versa,   and 
there are many fascinating connections between these areas. 
Examples are  the relation between the entropy power inequality and the Brunn-Minkowski inequality (see, e.g. \cite{Gardner2002}), the connection, established in \cite{NagySchuettWerner},  between the floating body \cite{BaranyLarman1988, SW1990} from convex geometry and data depth from statistics (see also  \cite{Brunel}, the  relation between the $L_p$-affine surface area and  R\'enyi entropy from information theory and statistics \cite{PaourisWerner2011, Werner2012/1, Werner2012} and connections between convex geometry and quantum information theory (e.g., \cite{Aubrun-SzarekBook, Aubrun-Szarek-Werner2010, Aubrun-Szarek-Werner2011, Aubrun-Szarek-Ye2012, Aubrun-Szarek-Ye2014, Szarek-Werner-Zyczkowski}).
Further examples can be found in the books by Cover, Dembo, and Thomas \cite{DCT} and Villani \cite{Villani} and in \cite{GLYZ, JenkinsonWerner, LutwakYangZhang2002/1, LutwakYangZhang2004/1, LutwakYangZhang2005}.
\par
\noindent
The $L_p$-affine surface area  is a fundamental  notion in the theory of convex bodies. It has many remarkable properties. Aside from being linearly invariant, it is a valuation and it satisfies  an affine isoperimetric inequality: among all convex bodies with fixed volume, $L_p$-affine surface area is maximized (minimized, depending on $p$) by ellipsoids. Not surprisingly, it therefore  finds applications in e.g.,  affine differential geometry \cite{GaZ, LutwakOliker1995, Lutwak1996}, geometric flows \cite {Stancu, Trudinger-Wang, Wang}, valuation theory \cite{Colesanti-Mussnig-Ludwig-1, HabSch2,  Ludwig-Reitzner1999,  Ludwig-Reitzner, Schuster2010, SchusterWannerer} and approximation theory 
\cite{Boeroetzky2000, Boeroetzky2000/2, BoeroetzkyReitzner2004, GTW2018, GroteWerner2018, Reitzner, SchuettWerner2003}. For extensions to the spherical and hyperbolic setting 
see \cite{Besau-Werner2016, Besau-Werner2018}.
\par
\noindent
For   a convex body $K$ in  $\mathbb R^n$ 
and real $p \ne -n$,  it is defined as 
\begin{equation} \label{def:p-affine}
as_{p}(K)=\int_{\partial K}  \frac{ \kappa_K(x)^{\frac{p}{n+p}}   }{ \langle x , N_{K} (x) \rangle^{\frac{n(p-1)}{n+p}} } \
 d\mu_{ K}(x),
\end{equation}
where $\mu_K$ is the usual surface measure on $\partial K$, the boundary of $K$, $ N_{K} (x) $ is the outer unit normal vector at $x$ to $\partial K$ and $\kappa_K(x)$ is the (generalized) Gauss curvature in $x \in \partial K$. The case $p=1$ is the classical affine surface area introduced by 
Blaschke \cite {Blaschke} in dimensions $2$ and $3$ for sufficiently smooth bodies  and extended much later  to all convex bodies by Leichtweiss \cite{Leichtweiss:1986a}, 
Lutwak \cite{Lutwak1991} and Sch\"utt and Werner \cite{SW1990}.
Then, Lutwak, in his ground breaking paper \cite{Lutwak1996}, introduced $L_p$-affine surface area for $p >1$. It was  finally  extended to all $p\in \mathbb{R}$, $p \neq -n$, and all convex bodies  in \cite{SW2004}, (see also \cite{HanSlomkaWerner, MW2}). 
\par
\noindent
In \cite{Werner2012/1} it was shown that $L_p$-affine surface areas are R\'enyi entropies. The latter are specific examples of $f$-divergences.
An $f$-divergence, which  is  an important concept from information theory, is a function that measures the difference between (probability)
distributions  $P$ and $Q$ (see, e.g. \cite{AliSilvery1966, Csiszar, Morimoto1963}). Further examples of  $f$-divergences are  Kullback-Leibler divergence \cite{KullbackLeibler1951} 
$D= D_{KL} (P||Q)$
and total variation distance $V = V(P,Q)$.
For various purposes it is of interest to investigate if they can be compared to one another. The most famous such comparison inequality is Pinsker's inequality \cite{Pinsker1960} which states that 
\begin{equation} \label{PinskerIneq}
 D \geq \frac{1}{2} V^2 .
\end{equation}
The best constant, $\frac{1}{2}$, is due, 
independently to Csisz\'ar \cite{Csiszar1967}, Kemperman \cite{Kemperman1969} and Kullback \cite{Kullback1967, Kullback1970}.
For  applications of Pinsker's inequality, see e.g., \cite{Barron1986}, \cite{Csiszar1984}, \cite{Topsoe1979}, or 
Bolley and Villani \cite{Bolley-Villani}  who showed that  Pinsker's  inequality is a  variant of Talagrand's transportation inequality.
Generalizations of  Pinsker type inequalities for $f$-divergences were obtained by  G. Gilardoni \cite{Gilardoni2010},   M. Reid and R. Williamson \cite{Reid-Williamson2009}. 
\par
\noindent 
In recent years much effort has been devoted to extend concepts from convex geometry to a functional  setting. Natural analogs of  convex bodies are log-concave functions. Much progress has been made in this direction, resulting in
functional analogs of the  Blaschke Santal\'o inequality \cite{ArtKlarMil, KBallthesis,  FradeliziMeyer2007, Lehec2009},  the affine isoperimetric inequality
\cite{ArtKlarSchuWer, CFGLSW}, Alexandrov-Fenchel type inequalities \cite{CaglarWerner2}   and analogs of the John ellipsoid \cite{Alonso-Gutierrez2017} and L\"{o}wner ellipsoid \cite{LSW2019}. More examples can be found in e.g.,  \cite{Alonso-Gutierrez, Alonso-Gutierrez2016, CaglarYe, Colesanti, Colesanti-Fragala, Gardner2002, 
KoldobskyZvavitch, Rotem}).
 In particular, $L_\lambda$-affine surface area, the functional analog of  $L_p$-affine surface area for convex bodies, was introduced in  
 \cite{CFGLSW} for a log concave function  $\varphi(x)= e^{-\psi(x)}$ and $\lambda \in \mathbb{R}$,
\begin{eqnarray*}
as_\lambda(\varphi) &=&
\int_{}  \varphi \   \left(   
\frac{e^{\frac{\langle\grad\varphi, x \rangle}{\varphi}}}{\varphi^2}  \  \mbox{det} \left[   \nabla^2  \left(- \ln \varphi \right)\right] \right)^\lambda dx,
\end{eqnarray*}
where $\grad \varphi$  is the gradient and 
$\nabla^2  \varphi$ is the (generalized) Hessian of $\varphi$.
In that context, entropy inequalities for log-concave functions were established. We only mention a reverse log-Sobolev inequality, proved in \cite{ArtKlarSchuWer, CFGLSW}, 
\begin{eqnarray}\label{Gleichung1Intro}
 \int_{}   \ \ln \bigg(\det \left( \nabla^2  \psi \right)\bigg)  e^{- \psi(x) } dx  &\leq& 2\left[ \operatorname{Ent}(\varphi) - \operatorname{Ent}(g) \right],
\end{eqnarray}
where  $\varphi:\R^{n}\rightarrow [0, \infty)$  is  the density of a  log concave
probability measure on $\mathbb{R}^n$, and $\operatorname{Ent}(\varphi)$, resp. 
$\operatorname{Ent}(g)$ the entropy of $\varphi$ resp. the Gaussian $g$. The  definitions are given below.
\par
\noindent
In \cite{CaglarWerner}, a theory of $f$-divergences  for log-concave functions was initiated. For instance, Kullback-Leibler divergence and related  inequalities for log-concave functions was established as part of the theory.  It already 
yielded entropy inequalities which are stronger than  already existing ones. For example, it resulted in the following  strengthing of  the reverse log-Sobolev  inequality (\ref{Gleichung1Intro}) of \cite{ArtKlarSchuWer}, 
\begin{eqnarray}\label{Gleichung2Intro}
 \int_{}   \ \ln \bigg(\det \left( \nabla^2  \psi \right)\bigg)  e^{- \psi(x) } dx  &\leq& 2\left[ \operatorname{Ent}(\varphi) - \operatorname{Ent}(g) \right] 
+ \ln \left( \frac{\int_{} e^{- \psi^*} }{(2  \pi)^n}\right), 
\end{eqnarray}
where  $\varphi:\R^{n}\rightarrow [0, \infty)$  is  the density of a  log concave
probability measure on $\mathbb{R}^n$,   and $\psi^*$ is the Legendre transform of $\psi$.
\vskip 3mm
\noindent
In this paper we further develop  the theory of $f$-divergences for log-concave  functions and their related inequalities. 
We establish a Pinsker  inequality and new affine invariant entropy inequalities for log-concave functions.   We obtain new inequalities on  functional affine surface area for log-concave functions and  lower and upper bounds for the Kullback-Leibler divergence  in terms of functional affine surface area.  The inequalities obtained for log-concave functions lead to new inequalities for  $L_p$-affine surface areas  for convex bodies. 
\par
\noindent
We start the paper by characterizing the equality case of inequality (\ref{Gleichung2Intro}) in Section \ref{MAGleichung}. While equality characterizations of inequality  (\ref{Gleichung1Intro}) were provided in \cite{CFGLSW},  no such characterizations were available up to date for inequality (\ref{Gleichung2Intro}) and for more general  $f$-divergence inequalities.
We show first that equality characterization is equivalent to uniqueness of the solution of  a Monge Amp\`ere differential equation (also called  elliptic K\"ahler-Einstein equation). Then we show that the Monge-Amp\`ere equation has a unique
solution. To do that we use optimal  transportation and Cafarelli's regularity theory for  optimal  transportation. 
\par
\noindent
In Section \ref{Section-Pinsker}, we show that  a Pinsker type inequality  for log concave functions follows immediately from a result  by G. Gilardoni \cite{Gilardoni2010}.
Namely, for a convex function 
 $f: (0, \infty) \rightarrow \R$,  we have
\begin{eqnarray*}\label{GillardoniLOG}
D_f(\varphi) \ \geq \ \frac{f''(1)}{2}  
\left( \int \left| \frac{ e^{\frac{\langle\grad\varphi, x \rangle}{\varphi}} \mbox{det} \left[ \nabla^2  \left(-\ln \varphi \right)\right] }{\varphi \int e^{- \psi^*} } \ - \ \frac{\varphi}{\int \varphi} \right| dx \right)^2 , 
\end{eqnarray*}
where $D_f(\varphi) $ is the $f$-divergence of the  log concave function $\varphi$ (see Section \ref{SectionLogcon} for the  definition).
A consequence of this Pinsker  inequality is  an improvement of inequality (\ref{Gleichung2Intro}), which takes the form (see Corollary \ref{PinskerIneqFunctional} for the precise statement), 
\begin{eqnarray*}\label{eqn0:Korollar1}  
 \hskip -5mm  \int_{}   \ \ln \bigg(\det \left( \nabla^2  \psi \right)\bigg)  e^{- \psi(x) } dx  &\leq& 2\left[ \operatorname{Ent}(\varphi) - \operatorname{Ent}(g) \right] 
+ \ln \left( \frac{\int_{} e^{- \psi^*} }{(2  \pi)^n}\right)   \nonumber  \\
& - &\frac{1}{2}
\left(
 \int_{}  \left| \frac{ e^{\psi - \langle \grad \psi, x \rangle} \mbox{det} \left (\nabla^2 \psi \right) } { \int_{} e^{-\psi^*} } 
\ - \   e^{-\psi}  \right| dx 
\right)^2.
\end{eqnarray*}
\par
\noindent
In Section \ref{SectionResults1} we prove  difference inequalities for  functional affine surface areas and 
show that $\lambda$-affine surface area of a log concave function $\varphi$ and its polar $\varphi^\circ$ (see (\ref{Polare}) below  for the  definition) is bounded  by a convex combination of $0$-affine surface area and $1$-affine surface area.   We  obtain lower and upper bounds for the Kullback-Leibler divergence $D_{KL} (Q_{\varphi}||P_{\varphi})$ in terms of functional affine surface area,  
$$ as_0 (\varphi) - as_1 (\varphi)  \ \leq D_{KL} (Q_{\varphi}||P_{\varphi}) \leq \ as_{-1} (\varphi) - as_0 (\varphi)
$$
where $P_{\varphi}$ and $Q_{\varphi}$ are two distributions with densities  $ \ p_\varphi= \varphi^{-1}  e^{\frac{\langle\grad\varphi, x \rangle}{\varphi}} \mbox{det} \left[ \nabla^2  \left(-\ln \varphi \right)\right]$ and  $ \ q_\varphi=\varphi $, respectively.  Please see (\ref{KLdiv-Logconcave}) for the exact definition of $ D_{KL} (Q_{\varphi}||P_{\varphi})$. 

\par
\noindent
We show in Section \ref{Applications} that on the level of convex bodies these inequalities  correspond to inequalities on $L_p$-affine surface area,  e.g., the following ones, 
\begin{equation*}
 as_{\infty} (K) -  as_{0} (K)  \  \leq \   as_{\frac{p}{n+p}}(K)   -  as_{\frac{n}{n+p}} (K), 
\end{equation*}
for $ p \in (-\infty , -n)$, and for $ p > -n$, the inequality is reversed. And
\begin{equation*} 
 as_p (K) \ \leq \ \left(\frac{p}{n+p} \right)  as_{\infty} (K) +  \left( \frac{n}{n+p} \right)
 as_0 (K), 
\end{equation*}
in the case when  $ p >0$.  For $ p <0$, those  inequalities are reversed. Equality holds trivially if $p=0$ or $p = \infty$.
Equality also holds for origin symmetric ellipsoids 
$\E$ whose volume $|\E|$ equals the volume $|B^n_2|$ of the Euclidean unit ball $B^n_2$.

\vskip 3mm
\noindent

\section{A Monge-Amp\`ere equation and equality  in a divergence inequality}\label{MAGleichung}

\subsection{Background on $f$-divergence}\label{section-Background on $f$-divergence}
Csisz\'ar \cite{Csiszar}, and independently Morimoto \cite{Morimoto1963} and Ali \& Silvery \cite{AliSilvery1966} introduced 
the notion of $f$-divergence to measure the difference between  probability
distributions.
This notion  finds applications in e.g.  information theory, statistics,  probability theory, signal processing, and pattern recognition \cite{BarronGyorfiMeulen, CoverThomas2006, HarremoesTopsoe, LieseVajda1987,  LieseVajda2006, OsterrVajda}.
\vskip 2mm
\noindent
Let $(X, \mu)$ be a measure space  and let  $P=p \mu$ and  $Q=q \mu$ be  (probability) measures on $X$ that are  absolutely continuous with respect to the measure $\mu$.  
Let $f: (0, \infty) \rightarrow  \mathbb{R}$ be a convex  or a concave  function.
Then the {\em $f$-divergence   $D_f(P,Q)$}  of the measures $P$ and $Q$ is defined by 
\begin{equation}\label{def:fdiv2}
D_f(P,Q)=
 \int_{X} f\left(\frac{p}{q} \right) q d\mu.
\end{equation}
\vskip 2mm
\noindent
The best known examples of $f$-divergences are the {\em total variation distance}
$$ V(P,Q) = \int |p-q| \ d \mu  \ \  \text{for \ } f(t) = |t-1| ,$$
and the {\em Kullback-Leibler divergence}, or {\em relative entropy}
\begin{equation} \label{K-L-divergence}
 D_{KL} (P||Q) = \int p \log \left(\frac{p}{q}\right) \ d \mu \ \ \text{for \ } f(t) = t \log  t .
\end{equation}
We also note that for $f(t) = t^{\alpha}$ we obtain the { \em Hellinger integrals} (see, e.g.,  \cite{LieseVajda2006})
$$
H_{\alpha} (P,Q) = \int_{X} p^{\alpha} q^{1- \alpha} \ d\mu. 
$$
Those are related to the 
{ \em R\'enyi divergence} of order $\alpha$, $\alpha \neq 1$,  introduced by  R\'enyi \cite{Ren} (for $\alpha >0$) as 
\begin{equation}\label{renyi}
D_\alpha(P\|Q)=
\frac{1}{\alpha -1} \log \left( \int_X p^\alpha q^{1-\alpha} d\mu \right)= \frac{1}{\alpha -1} \log \left( H_\alpha (P,Q)\right).
\end{equation}
The case $\alpha =1$ is the relative entropy $ D_{KL}(P\|Q)$.
\vskip 1cm

\subsection{ $f$-divergence for log concave functions}\label{SectionLogcon}

 Let $\psi: \R^n \rightarrow \R \cup \{ \infty \}$
 be a  convex function.  We define $\Omega_\psi$ to be the interior of the convex domain of $\psi$, that is 
\[
\Omega_\psi= \mathrm{int} \, (\{x \in \R^n, \psi(x) < \infty\} )=  \mathrm{int} \, (\{\psi < \infty\} ).
\]
\par
\noindent
We always consider in this paper convex  functions $\psi$ such that $\Omega_\psi \ne \emptyset$. Then $\psi$ is in particular $\psi$ is proper, i.e., $\psi(x)$ is finite for at least one $x$ and that the affine dimension of $\Omega_\psi$ is equal to $n$.  
This implies that 
\begin{equation}\label{positive}
 \int _{\Omega_\psi} e^{-\psi(x)} dx >0. 
\end{equation}
We will also assume throughout  that $e^{-\psi(x)}$ is integrable,  i.e., 
$ \int _{\Omega_\psi} e^{-\psi(x)} dx < \infty$  and we will also often write  in short $ \int _{} e^{-\psi(x)} dx$.
\newline
In the general case, when $\psi$ is neither  smooth nor strictly convex, 
the gradient of $\psi$, denoted  by $\nabla \psi$, exists almost everywhere by Rademacher's theorem (see, e.g., \cite{Rademacher}),   and a theorem of Alexandrov \cite{Alexandroff} and Busemann and Feller \cite{Buse-Feller} guarantees the existence of the (generalized) Hessian, denoted  by
$\nabla^2 \psi$, almost everywhere in $\Omega_\psi$.  Let
 $$X_\psi=\Big\{x\in \R^n:\ \psi(x)<\infty,\ \text{and}\  \nabla^2\psi(x)\ \text{exists\  and\  is\ invertible}\Big\}.$$ 
\par
\noindent
We recall the  Legendre transform $\mathcal{L}\psi$ of $\psi$, 
\begin{equation} \label{Legendre}
\mathcal{L}\psi(y) = \psi^*(y) = \sup_{x \in \R^n} \left[ \langle x,y \rangle  -\psi(x) \right].
\end{equation}
When $\psi $ is $C^2$,  then $X_\psi = \Omega_\psi$ and $X_{\psi^*} = \Omega_{\psi^*}$. 
More  information about duality transforms  of convex functions can be found in \cite{Rockafellar, SchneiderBook}.
\vskip 2mm
\noindent
 A function $\varphi: \mathbb{R}^n \rightarrow [0, \infty)$ is   log concave, if it is of the form $\varphi(x)= \exp(-\psi (x) )$ where  $\psi: \R^n \rightarrow
 \R \cup \{ \infty \}$ is convex. 
The dual function  $\varphi^\circ$  of a log concave function is defined by  \cite{ArtKlarMil, KBallthesis} 
$$ \ \varphi^\circ (x) = \inf_{y \in \R^n } \left[ \frac{e^{- \langle x,y \rangle}}{\varphi(y)}  \right].$$
This definition is connected with the Legendre transform,  
namely, 
\begin{equation}\label{Polare}
 \varphi^\circ  = e^{-\mathcal{L}\left( \psi \right) } = e^{-\psi^*}.
\end{equation}
 \par
 \noindent
 In \cite{CaglarWerner}, 
 $f$-divergences for $s$-concave and log concave functions were introduced and their  basic properties and entropy inequalities  were established.  
It  is explained in detail in \cite{CaglarWerner} that the following definition for $f$-divergence seems to be the correct one.
\begin{defn} \label{defi2}
Let $f: (0, \infty) \rightarrow \mathbb{R}$ be  a  convex or concave function and let $ \varphi:\R^{n}\rightarrow [0, \infty)$ be a log concave function.
Then the $f$-divergence $D_f(\varphi)$ of $\varphi$ is 
\begin{eqnarray}\label{div-Logconcave1}
D_f(\varphi) &=& 
 \int_{X_\psi} \varphi  \ f  \left(
\frac{e^{\frac{\langle\grad\varphi, x \rangle}{\varphi}} }{\varphi ^2} \  \mbox{det} \left[ \nabla^2  \left( -\ln \varphi \right)\right] \right)dx  \nonumber \\
&=&  \int_{X_\psi}  e^{-\psi}  \ f  \left(
e^{2 \psi - \langle\grad\psi, x \rangle}\  \mbox{det} \left[  \nabla^2 \psi \right] \right)dx.
\end{eqnarray}
\end{defn}
\par
\noindent
The special case when $f(t)=t^\lambda$, $- \infty < \lambda <  \infty$ leads to the $L_\lambda$-affine surface areas of $\varphi$ \cite{CFGLSW}, 
\begin{eqnarray}\label{asp-Logconcave}
as_\lambda(\varphi) &=&
\int_{X_\psi} \varphi \   \left(   
\frac{e^{\frac{\langle\grad\varphi, x \rangle}{\varphi}}}{\varphi^2}  \  \mbox{det} \left[  \nabla^2  \left(- \ln \varphi \right)\right] \right)^\lambda dx \nonumber \\
&=& \int_{X_\psi} e^{(2\lam-1)\psi(x)-\lam\langle x, \nabla\psi(x)\rangle}\left(\det \, {\nabla^2 \psi (x)}\right)^\lam dx.
\end{eqnarray}
\noindent
Those were extensively studied in \cite{CFGLSW}. 
In particular,  
\begin{equation}\label{volumephi}
as_0(\varphi)= \int_{X_\psi} \varphi dx
\end{equation}
and, as observed in  \cite{CFGLSW},  since $\mbox{det} \left[ \nabla^2  \left(-\ln \varphi \right)\right] = 0$ outside of $X_\psi$, the integral may be taken over $\Omega_{\psi}$ for any $\lambda >0$.  Therefore
\begin{eqnarray}\label{sublimit}
as_{1}(\varphi)  &=&  \int_{X_{\psi}}  \varphi^{-1}  e^{\frac{\langle\grad\varphi, x \rangle}{\varphi}} \mbox{det} \left[ \nabla^2  \left(-\ln \varphi \right)\right]  dx 
 = \int_{\Omega_{\psi}}  \varphi^{-1}  e^{\frac{\langle\grad\varphi, x \rangle}{\varphi}} \mbox{det} \left[ \nabla^2  \left(-\ln \varphi \right)\right] dx   \nonumber \\
 &=& \int_{X_{\psi^*}} \varphi^\circ.
\end{eqnarray}
In analogy to (\ref{volumeK}) below,  the expressions (\ref{div-Logconcave1})  are the appropriate ones to define  $f$-divergences for log concave  functions
and  because of (\ref{volumephi}) and (\ref{sublimit}) these expressions 
can be viewed as the ``volume" of $\varphi$ and the ``volume" of $\varphi^\circ$ with 
their corresponding ``cone measures".
This is explained in detail in \cite{CaglarWerner}.
\par
\noindent
Another special case  occurs when $f(t)=-\ln t$. The  $f$-divergences then  becomes the Kullback-Leibler divergence 
\begin{eqnarray}\label{KLdiv-Logconcave}
D_{KL}(Q_\varphi || P_\varphi) &=&
 \int_{X_\psi}  \varphi  \ \ln  \left( \varphi ^2 
e^{-\frac{\langle\grad\varphi, x \rangle}{\varphi}} \  \left(\mbox{det} \left[ \nabla^2  \left( -\ln \varphi \right)\right] \right)^{-1} \right) dx \nonumber  \\
&=&  \int_{X_\psi}  e^{-\psi}   ( \langle\grad\psi, x \rangle  -2 \psi ) \    \ln     \left(\mbox{det} \left[ \nabla^2  \left(  \psi \right)\right] \right)^{-1}  dx,
\end{eqnarray}
where $P_{\varphi}$ and $Q_{\varphi}$ are two distributions with densities  $ \ p_\varphi= \varphi^{-1}  e^{\frac{\langle\grad\varphi, x \rangle}{\varphi}} \mbox{det} \left[ \nabla^2  \left(-\ln \varphi \right)\right]$ and  $ \ q_\varphi=\varphi $, respectively.

\subsection{A Monge-Amp\`ere equation}\label{}

Now we concentrate on the following divergence inequality, which was 
also proved in \cite{CaglarWerner}. Recall that we assume throughout the paper that $ \int _{} e^{-\psi(x)} dx < \infty$.
\par
\noindent
\begin{theo} \cite{CaglarWerner} \label{thm00}
 Let $f: (0, \infty) \rightarrow \mathbb{R}$ be a convex function.  Let $ \varphi:\R^{n}\rightarrow [0, \infty)$ be a
integrable  log-concave
function that is $C^2$.
Then
\begin{eqnarray}\label{thm00,1}
D_f(\varphi)
\geq \ f \left(  \frac{\int_{\Omega_{\psi^*}}  \ \varphi^\circ dx}{ \int_{\Omega_{\psi}} \varphi dx}   \right)
 \  \left( \int_{\Omega_{\psi}} \varphi  dx \right).
 \end{eqnarray}
If  $f$ is concave, the inequality is reversed. 
If $f$ is linear, equality holds in (\ref{thm00,1}).  
Equality also holds 
if
 $\varphi(x)=C e^{-\langle A x, x \rangle}$, where
 $C$ is a positive constant  and $A$ is an $n \times n$ positive definite  matrix.
 \end{theo}
 \par
 \noindent
A characterization of the equality case of inequality (\ref{thm00,1}) - and several other inequalities proved  in  \cite{CaglarWerner} -  has not been obtained so far.  Here we give such a characterization.
We show first that characterization of equality in (\ref{thm00,1}) is equivalent to the unique solution of a Monge Amp\`ere differential equation.
\par
\noindent
We write $\varphi = e^{-\psi}$. It was shown in \cite{CaglarWerner} that 
 inequality \eqref{thm00,1} is a consequence of Jensen's inequality and  the identity (see, e.g., \cite{CaglarWerner}), 
 \begin{equation} \label{transformation}
\int _{\Omega_{\psi^*}}  e^{-\psi^* }dx  \ = \  \int  _{\Omega_\psi}  e^{ \psi - \langle \nabla \psi,x\rangle}   \   \det \left( \nabla^2 \psi \right) dx. 
\end{equation}
Thus,  equality holds in \eqref{thm00,1} if and only if equality holds in
 Jensen's inequality which happens if and only if  
\begin{equation}
\det(\nabla^2\psi(x))=C\, e^{-2\psi(x)+\langle \nabla \psi(x),x\rangle},\qquad  \mbox{a.e.}\;  x\in  {\R}^n .\label{eql21}
\end{equation}
To determine $C$,  we  integrate (\ref{eql21}) to get
$$C \int_{\Omega_{\psi}} e^{-\psi(x)}\, dx = \int_{\Omega_{\psi}}  e^{ \psi - \langle \nabla \psi,x\rangle}   \   \det \left( \nabla^2 \psi \right) dx,$$
which together with (\ref{transformation}) gives that 
$$C=\frac{\int_{\Omega_{\psi^*}} e^{-\psi^* }dx}{\int_{\Omega_{\psi}} e^{-\psi}dx }= \frac{\int_{\Omega_{\psi^*}} \varphi^\circ(x)\, dx}{\int_{X_\psi} \varphi(x)\, dx}. $$
Thus, when  $f$ is either strictly convex or strictly concave,  equality holds in  inequality  \eqref{thm00,1} if and only  if $\psi$ satisfies 
\begin{equation}
\det(\nabla^2\psi(x))=\frac{\int _{\Omega_{\psi^*}}  e^{-\psi^* }dx}{\int_{\Omega_{\psi}}  e^{-\psi}dx }\; e^{-2\psi(x)+\langle \nabla \psi(x),x\rangle},\qquad x\in  {\R}^n.\label{MA001}
\end{equation}
Recall now  that 
\[
\psi ( x) + \psi^* (y ) \geq \langle x,y \rangle
\] 
for every $x,y \in \R^n$, with equality if and only if
 $x$ is in the domain of $\psi$ and $y\in\partial\psi(x)$, the sub differential 
 of $\psi$ at $x$. In particular
\begin{equation}\label{legendreequality}
 \psi^* ( \nabla \psi (x) ) = \langle x ,\nabla \psi (x) \rangle - \psi (x)  ,\quad \rm{a.e.\ in}\ \Omega_\psi. 
\end{equation} 
Thus, equation (\ref{MA001}) can be written as
\begin{equation}\label{MA002}
\frac{  e^{-\psi} }{\int_{\Omega_{\psi} }    e^{-\psi}dx }  = \frac{e^{- \psi^*( \nabla \psi(x))}}{\int _{\Omega_{\psi^*}}  e^{-\psi^* }dx}  \  \det(\nabla^2\psi(x)),
\end{equation}
which is just a Monge Amp\`ere equation  (also called  elliptic K\"ahler-Einstein equation).
\par
\noindent
Note that if $\psi$ solves \eqref{MA002}, then it is not difficult to show that $\psi(x)+c$ solves \eqref{MA002} for any constant $c.$ Thus we  seek uniqueness of the  solution of \eqref{MA002} up to a constant and this is established the following  theorem.
\vskip 2mm 
\noindent
\begin{theo} \label{thm01} 
Let $\psi \colon \mathbb{R}^n \to \mathbb{R} \cup \{+\infty\}$ be a  convex function such that
$$
e^{-\psi} dx  \ {\it and} \ e^{-\psi^*} dx
$$
are finite log concave measures.
Assume that 
 the  mapping $T(x) = \nabla \psi(x)$ pushes forward 
$$
d\mu = \frac{e^{-\psi}}{\int e^{-\psi} dx} dx  \ \
{\it onto}
\ \
d\nu =  \frac{e^{-\psi^*}}{\int e^{-\psi^*} dx} dx.
$$
In addition, assume that $\mu$ has logarithmic derivatives for every $x_i$, $1 \le i \le n$.
Then $\psi$ has the form $\psi = \frac{1}{2} \langle Ax, x \rangle + c$, where $c$ is a constant and $A$
is a $n \times n$ positive definite matrix.
\end{theo}
\vskip 2mm
\noindent
Before we prove this theorem we need to establish some preliminary results concerning
integral relations for solutions to the optimal transportation problem. For more detail and background we refer to \cite{Villani}.
\newline
Let $\mu= e^{-V} dx$ be a probability measure on $\mathbb{R}^n$. 
We say that $V_{x_i}$ is a logarithmic (Sobolev) derivative of $\mu$ if $V_{x_i} \in L^{1}(\mu)$
and  for every compactly supported smooth test function $\xi$ the following relation holds, 
$$
\int \xi_{x_i} d \mu = \int \xi V_{x_i} d \mu.
$$
\par
\noindent
As noted above, in the case of a log concave measure $\mu = e^{-V}$ the function $V$ is almost everywhere differentiable on
$\Omega_V= \mathrm{int} \, (\{ V < \infty\} )$, but this does not mean that  $\mu$ has logarithmic derivatives. Indeed,
in general the integration by parts formula includes a  singular term,
$$
\int \xi_{x_i} d \mu = \int \xi V_{x_i} d \mu + \int_{\partial \{V < \infty\}} \langle n, e_i\rangle \xi e^{-V} 
d \mathcal{H}^{n-1},
$$ 
where $n$ is the outward normal vector to $\partial \{ V < \infty\}$ and $\mathcal{H}^{n-1}$ is the $(n-1)$-dimensional Hausdorff measure.
Thus $\mu$ does not admit a logarithmic derivative if $\{ V = \infty\}$ is not empty and
$e^{-V}$ is not vanishing on $\partial \{ V < \infty\}$.
\vskip 2mm
\noindent
In what follows,  we are given two probability measures $d\mu= e^{-V} dx$ and  $d\nu= e^{-W} dx$. Let $\nabla \psi$ be the optimal transportation of  $\mu= e^{-V} $ onto  $\nu= e^{-W}$.
We remark that $\psi_{x_i}$ is always understood in the classical sense, i.e. almost everywhere pointwise. The next proposition was proved in  \cite{KK}.
\par
\noindent
\begin{prop} ({\bf Proposition 5.5.} \cite{KK}) \label{KKpropos}
Assume that $V, W, \psi$ are smooth functions on the entire $\mathbb{R}^n$ and $\nu$ is a log concave measure. Then  for every $q \geq 2$, $0 <\tau <1$, $i=1, \cdots, n$ there exists $C(q, \tau) >0$ such that 
$$
\int_{\mathbb{R}^n}  |\psi_{x_i x_i}|^{q} d \mu \leq C(q,\tau) \left( \int_{\mathbb{R}^n} |V_{x_i}|^{\frac{2q}{2-\tau}} d \mu + \int_{\mathbb{R}^n} |x_i|^{\frac{2q}{\tau}} d \nu \right).
$$
\end{prop}
\par
\noindent
In the case when $\mu$ admits logarithmic derivatives which are integrable in a sufficiently high power,  it is natural to understand the second derivatives of $\psi$ in the Sobolev sense. More precisely, we say that $\psi$ admits second partial derivatives $\psi_{x_i x_j}$  in the Sobolev sense
if for every smooth compactly supported test function $\xi$, 
$$
\int \psi_{x_i x_j} \xi d \mu = - \int \xi_{x_i} \psi_{x_j} d \mu + \int V_{x_i} \psi_{x_j} \xi d \mu.
$$
We will use throughout  differentiation of the Monge-Amp\`ere equation developed in  \cite{KK} and \cite{Kol}.
Let us briefly explain this machinery. 
Differentiate the change of variables formula 
$$
V = W(\nabla \psi) - \log \det \nabla^2 \psi
$$ 
 along unit vector $e$, 
\begin{equation}
\label{Ve}
V_e = \langle \nabla^2 \psi_e, \nabla W(\nabla \psi) \rangle - {\rm Tr} (\nabla^2 \psi)^{-1} \nabla^2 \psi_e. 
\end{equation}
Introduce  the second-order  differential operator $L$, 
$$
L f =   {\rm Tr} (\nabla^2 \psi)^{-1} \nabla^2 f - \langle \nabla f, \nabla W(\nabla \psi)\rangle.
$$
One can verify (see \cite{Kol}, Lemma 2.1) that measure $\mu$ is invariant with respect to $L$
$$
\int L \xi \cdot \eta d \mu= - \int \langle (\nabla^2 \psi)^{-1} \nabla \xi, \nabla \eta \rangle d \mu, 
$$
where $\xi$ and $\eta$ are  smooth functions with compact supports 
$$K_{\xi}, K_{\eta} \subset \{ V < \infty\}.$$
Equation (\ref{Ve}) can be rewritten as follows:
$$
V_e = - L \psi_e.
$$
Differentiating second time one gets
\begin{equation}
\label{vee=-L}
V_{ee} = - L \psi_{ee} + \langle \nabla^2 W(\nabla \psi), \nabla \psi_e, \nabla \psi_e \rangle + {\rm Tr}\bigl[ (\nabla^2 \psi)^{-1}(\nabla^2 \psi_e)\bigr]^2 .
\end{equation}
Integrating with respect to $\mu$ and using invariance of $\mu$  one gets formally
the following integral identity obtained in \cite{Kol2010}
\begin{equation}
\label{integral2}
\int V^2_e d \mu = \int V_{ee} d \mu =  \int \langle \nabla^2 W(\nabla \psi) \nabla \psi_e, \nabla \psi_e \rangle d \mu + \int {\rm Tr}\bigl[ (\nabla^2 \psi)^{-1}(\nabla^2 \psi_e)\bigr]^2 d \mu.
\end{equation}
Note that $\int V^2_e d \mu = \int V_{ee} d \mu $ as $\mu$ admits logarithmic derivatives.
\par
\noindent
We stress that it is indeed a formal relation, because we neglect boundary terms which may appear. Formula (\ref{integral2}) holds under additional assumptions on the growth and smoothness of $V,W$.
\par
\noindent
Now  we will apply the following slight extension of Proposition 5.5. of  \cite{KK} (see Remark 5.6 in \cite{KK}), which can be easily obtained by smooth approximations.
\begin{prop} { \cite{KK}}
\label{sobolev-reg}
Assume that $\nu$ is a log concave measure and that  $\mu$ admits logarithmic derivatives $V_{x_i}$ for all $1 \le i \le n$ which are integrable in any power.
Then $\psi$ admits second Sobolev derivatives with respect to $\mu$  satisfying
\begin{equation}
\label{d2p}
\int |\psi_{x_i x_i}|^{q} d \mu \leq C(q,\tau) \left( \int |V_{x_i}|^{\frac{2q}{2-\tau}} d \mu + \int |x_i|^{\frac{2q}{\tau}} d \nu \right),
\end{equation}
where $q \ge 2,  0 < \tau <1$.
\end{prop}
\par
\noindent

%We will apply the following (formal) integral identity obtained in \cite{Kol2010}. There, $e$ is any unit vector in $\mathbb{R}^n$.  Note that $\int V^2_e d \mu = \int V_{ee} d \mu $ as $\mu$ admits logarithmic derivatives.
%\begin{equation}
%\label{integral2}
%\int V^2_e d \mu = \int V_{ee} d \mu =  \int \langle \nabla^2 W(\nabla \psi) \nabla \psi_e, \nabla \psi_e \rangle d \mu + \int {\rm Tr}\bigl[ (\nabla^2 \psi)^{-1}(\nabla^2 \psi_e)\bigr]^2 d \mu.
%\end{equation}
%\par
%\noindent

Inequality (\ref{integral2}) follows from the integration by parts formula applied to an identity obtained
by differentiation of the Monge--Amp\'ere equation.
Our next aim is to justify (\ref{integral2}) in form of inequality in a sufficient general setting.
\par
\noindent
\begin{prop}
\label{l2ineq}
Let $\mu = e^{-V} dx$   and $\nu = e^{-W} dx$ be  probability measures, $V \colon \mathbb{R}^n \to \mathbb{R} \cup 
\{+ \infty\}$, $W \colon \mathbb{R}^n \to \mathbb{R} \cup \{+ \infty\}$.
Let $\nabla \psi$ be the optimal transportation mapping of $\mu$ onto $\nu$. 
\par
\noindent
Assume that the following assumptions hold.
\begin{enumerate}
\item
The sets $\{ V < \infty\}$ and  $\{W < \infty\}$ are open and
$V$ and $W$ are twice   differentiable  on 
 the sets $\{V<\infty\}$ and $\{W < \infty\}$ 
respectively, 
with locally H{\"o}lder second derivatives.
\item
The measure $\mu$ admits  logarithmic derivatives $V_{x_i}$  which are integrable in every power with respect to $\mu$
$$
\int |V_{x_i}|^p d\mu < \infty, \ \forall p >0.
$$
\item
$\nu$ is log-concave.
\end{enumerate}
Then $\psi$ is at least four times continuously differentiable 
on  $\{V<\infty\}$ 
and the following integral inequality holds for every unit vector $e\in \mathbb{R}^n$,
\begin{equation}
\label{l2-ineq}
\int_{\{V < \infty\}} V_{ee} d \mu \ge  \int \langle \nabla^2 W(\nabla \psi) \nabla \psi_e, \nabla \psi_e \rangle d \mu + \int {\rm Tr}\bigl[ (\nabla^2 \psi)^{-1}(\nabla^2 \psi_e)\bigr]^2 d \mu.
\end{equation}
\end{prop}
\vskip 2mm
\noindent
\begin{proof}
Observe first that by assumption on $V$ and $W$, $\Omega_V=\{ V < \infty\}$ and $\Omega_W=\{ W < \infty\}$.
Next we note that by Proposition \ref{sobolev-reg}
$\int |\psi_{x_i x_i}|^p d \mu < \infty$ for every  $p>0$.
Using Sobolev embeddings and the fact that $V$ is locally bounded
we get that $\psi_{x_i}$ is locally Sobolev with respect to Lebesgue measure
on $\{V < \infty\}$ and continuous.
\par
\noindent
Let us show that $\psi$ is four times differentiable.
Since  $\nabla \psi$ is well defined and continuous almost everywhere, we can choose  $x_0 \in \{ V < \infty\}$  such that $\nabla \psi (x_0)$ exists.
Since $\{ W< \infty\}$ is open, there exists an open bounded convex neighborhood $U_2 \subset \{ W< \infty\}$ of $\nabla \psi(x_0)$.
By continuity  $U_1 = \nabla \psi^{-1}(U_2) \subset \{ V <\infty\}$ is an open set.  Moreover, since convex functions are locally Lipschitz, $U_1 =\nabla \psi^*(U_2)$  is bounded. 
 Then consider a mass transportation problem of $\mu|_{U_1}$
 onto $\nu|_{U_2}$. Thus $U_1$ and $U_2$ are bounded open sets, $U_2$ is convex and $\nabla \psi$ is the optimal transportation mapping of $\mu|_{U_1}$ onto $\nu|_{U_2}$. These measures have $C^{2,\alpha}$ densities for some $\alpha$,  $0 < \alpha \leq1$,  with respect to the Lebesgue measure.
Hence by Caffarelli's regularity theory, (see e.g., Theorem 4.14  and  Remark 4.15 from \cite{Villani}),  $\psi$ is locally $C^{4,\alpha}$ on $\{ V< \infty\}$.
\par
\noindent
Next we apply the following equation obtained by differentiation of the change of variables formula (see \cite{KK},  \cite{Kol}  for details), 
\begin{equation}
\label{vee=-L}
V_{ee} = - L \psi_{ee} + \langle \nabla^2 W(\nabla \psi) \nabla \psi_e, \nabla \psi_e \rangle + {\rm Tr}\bigl[ (\nabla^2 \psi)^{-1}(\nabla^2 \psi_e)\bigr]^2 .
\end{equation}
Here $L$ is the second-order  differential operator satisfying
$$
\int L \xi \cdot \eta d \mu= - \int \langle (\nabla^2 \psi)^{-1} \nabla \xi, \nabla \eta \rangle d \mu, 
$$
where $\xi$ and $\eta$ are  smooth test functions with compact supports $K_{\xi}, K_{\eta} \subset \{ V < \infty\}$.
\vskip 2mm
\noindent
{\bf Step 1.}  Assume that $\nu$ is fully supported, i.e., $W<\infty$ everywhere.
\par
\noindent
Take a smooth test function $\eta$ that has compact support in $\{ V < \infty\}$, multiply  (\ref{vee=-L}) by $ \xi=\eta(\nabla \psi)$ 
and integrate with respect to $\mu$.
One gets
\begin{align}\label{05052020}
\int_{\{V < \infty\}} V_{ee} \xi d \mu & = \int \langle (\nabla^2 \psi)^{-1} \nabla \psi_{ee}, \nabla (\eta(\nabla \psi)) \rangle d \mu \nonumber
\\&  + \int \langle \nabla^2 W(\nabla \psi) \nabla \psi_e, \nabla \psi_e \rangle \xi d \mu 
+ \int {\rm Tr}\bigl[ (\nabla^2 \psi)^{-1}(\nabla^2 \psi_e)\bigr]^2  \xi d \mu.
\end{align}
Note that
\begin{align*}
-\int \langle & (\nabla^2 \psi)^{-1} \nabla \psi_{ee}, \nabla (\eta(\nabla \psi)) \rangle d \mu 
 = -\int \langle  \nabla \psi_{ee}, \nabla \eta \circ \nabla \psi \rangle d \mu 
 \\& = -\int \langle  \bigl( \nabla^2 \psi \bigr) ^{-\frac{1}{2}} \nabla^2 \psi_{e} \cdot e, \bigl( \nabla^2 \psi \bigr) ^{\frac{1}{2}} \nabla \eta \circ \nabla \psi \rangle d \mu 
  \\& =  \int \langle A \bigl( \nabla^2 \psi \bigr) ^{\frac{1}{2}} \cdot e, \bigl( \nabla^2 \psi \bigr) ^{\frac{1}{2}} \nabla \eta \circ \nabla \psi \rangle d \mu
  \le \int \|A\| \|\nabla^2 \psi\| |\nabla \eta \circ \nabla \psi| d \mu,
 \end{align*}
 where $A =  (\nabla^2 \psi \bigr) ^{-\frac{1}{2}} \nabla^2 \psi_{e} (\nabla^2 \psi \bigr) ^{-\frac{1}{2}} $ and $\| \cdot \|$ is the operator norm.
 Next we note that
 $$
 \|A\|^2 \le \|A\|^2_{HS} = {\rm Tr} A^2 =  {\rm Tr}\bigl[ (\nabla^2 \psi)^{-1}(\nabla^2 \psi_e)\bigr]^2. 
 $$
 Hence for every $\varepsilon >0$
 \begin{align*}
-\int \langle & (\nabla^2 \psi)^{-1} \nabla \psi_{ee}, \nabla (\eta(\nabla \psi)) \rangle d \mu  
\\& \le \varepsilon \int   {\rm Tr}\bigl[ (\nabla^2 \psi)^{-1}(\nabla^2 \psi_e)\bigr]^2 \xi d \mu 
+ \frac{1}{4\varepsilon}  \int  \| \nabla^2 \psi \|^2 \frac{| \nabla (\eta(\nabla \psi)|^2}{\xi} d \mu
\\& = \varepsilon   \int {\rm Tr}\bigl[ (\nabla^2 \psi)^{-1}(\nabla^2 \psi_e)\bigr]^2  \xi d \mu
 + \frac{1}{4\varepsilon} \int \| \nabla^2 \psi \|^2 \frac{|\nabla \eta|^2}{\eta} \circ (\nabla \psi) 
d \mu 
\\& 
\le \varepsilon  \int {\rm Tr}\bigl[ (\nabla^2 \psi)^{-1}(\nabla^2 \psi_e)\bigr]^2  \xi d \mu
+
\frac{1}{4\varepsilon} \Bigl( \int \| \nabla^2 \psi \|^{2p} d \mu \Bigr)^{\frac{1}{p}} \Bigl( \int \frac{|\nabla \eta|^{2q}}{\eta^q}  d \nu \Bigr)^{\frac{1}{q}},
 \end{align*}
where $\frac{1}{p} +\frac{1}{q}=1$. 
Thus, for arbitrary $\varepsilon>0$
\begin{eqnarray} \label{l2-ineq-cs}
&&\int_{\{V < \infty\}}  V_{ee} \xi d \mu + 
\frac{1}{4\varepsilon} \Bigl( \int \| \nabla^2 \psi \|^{2p} d \mu \Bigr)^{\frac{1}{p}} \Bigl( \int \frac{|\nabla \eta|^{2q}}{\eta^q}  d \nu \Bigr)^{\frac{1}{q}} 
\nonumber
\\&&  \ge \int \langle \nabla^2 W(\nabla \psi) \nabla \psi_e, \nabla \psi_e \rangle \xi d \mu 
+ (1-\varepsilon) \int {\rm Tr}\bigl[ (\nabla^2 \psi)^{-1}(\nabla^2 \psi_e)\bigr]^2  \xi d \mu.
\end{eqnarray}
Finally, we want to extract  (\ref{l2-ineq}) from (\ref{l2-ineq-cs}).
To this end we  find a sequence of compactly supported functions $1 \ge \eta_n \ge 0$ with 
$\eta_n \to 1$ pointwise such that $\lim_{n} \int \frac{|\nabla \eta_n|^{2q}}{\eta^q_n} d \nu =0$ and set $\xi_n = \eta_n(\nabla \Phi)$.
We omit the description of the precise construction, since it  can be easily done taking into account that ${\rm supp}(\nu)  = \mathbb{R}^n$.
We apply inequality (\ref{l2-ineq-cs}), where $\xi$ is replaced by $\xi_n$ and pass to the limit letting $n$  to infinity.
Note that  $\int \| \nabla^2 \psi \|^{2p} d \mu < \infty$ by Proposition \ref{KKpropos}. 
Passing to the limit and applying that $\varepsilon>0$ is arbitrary,
one gets  (\ref{l2-ineq}).
Moreover, since the integrals
$\int \| \nabla^2 \psi \|^{2p} d \mu$, $\int {\rm Tr}\bigl[ (\nabla^2 \psi)^{-1}(\nabla^2 \psi_e)\bigr]^2  d \mu $
are  finite, it is clear that 
$$
 \int \langle (\nabla^2 \psi)^{-1} \nabla \psi_{ee}, \nabla (\eta_n(\nabla \psi)) \rangle d \mu \to 0
$$
and we have in fact equality in 
(\ref{l2-ineq}), because we can pass to the limit in (\ref{05052020}).
The proof of Step 1 is complete.
\vskip 2mm
\noindent
{\bf Step 2.}  Proof of  the general case: $W$ is twice H{\"o}lder differentiable on the open convex domain $\{ W <\infty\}$.
\par
\noindent
Approximate $W$ by everywhere finite and smooth convex functions $W_n$  such that 
$\nabla W_n \to \nabla W$ and $\nabla^2 W_n \to \nabla^2 W$  pointwise on $\{ W < \infty\}$.
This can be done with the standard convolution technique: set $e^{-W_n} = e^{-W} * \gamma_{\frac{1}{n}}$, where $\gamma_{\frac{1}{n}}$ is the Gaussian measure with zero mean and variance $\frac{1}{n}$.
By the Prekopa-Leindler inequality we get that every  $W_n$ is convex and, in addition,  smooth on the entire $\mathbb{R}^n$.
According to Step 1, 
\begin{align}
\label{vee}
\int_{\{V< \infty\}} V_{ee}  d \mu  & =
 \int \langle \nabla^2 W_n(\nabla \psi_n) \nabla (\psi_n)_e, \nabla (\psi_n)_e \rangle d \mu 
\nonumber \\& 
+ \int {\rm Tr}\bigl[ (\nabla^2 \psi_n)^{-1}(\nabla^2 (\psi_n)_e)\bigr]^2   d \mu, 
\end{align}
where $\nabla \psi_n$ is the optimal transportation mapping of $\mu$ onto $\nu_n=e^{-W_n} dx$. 
First we observe that it follows from (\ref{d2p}) that
\begin{equation}
\label{sobol-bound}
\sup_n \int \bigl( |\nabla \psi_n|^p + \| \nabla^2 \psi_n \|^p  \bigr) d \mu < \infty, \  \forall p>0.
\end{equation}
We may assume that 
$$
\int_{\{V < \infty\}} V_{ee} d \mu <\infty,
$$
as otherwise there is nothing to prove. Then one has by (\ref{vee})
\begin{eqnarray}\label{Ungleichung}
&&\hskip -15mm \int {\rm Tr}\bigl[ (\nabla^2 \psi_n)^{-1}(\nabla^2 (\psi_n)_e)\bigr]^2   d \mu = \int_{\{V< \infty\}} V_{ee}  d \mu  \nonumber \\
&&\hskip +10mm -  \int \langle \nabla^2 W_n(\nabla \psi_n) \nabla (\psi_n)_e, \nabla (\psi_n)_e \rangle d \mu  \leq \int_{\{V< \infty\}} V_{ee}  d \mu < \infty
\end{eqnarray}
and thus 
$$
\sup_n \int {\rm Tr}\bigl[ (\nabla^2 \psi_n)^{-1}(\nabla^2 (\psi_n)_e)\bigr]^2  d \mu < \infty.
$$ 
Hence
\begin{eqnarray*}
\infty &>& \int_{\{V < \infty\}} V_{ee} d \mu \geq \sup_n \int {\rm Tr}\bigl[ (\nabla^2 \psi_n)^{-1}(\nabla^2 (\psi_n)_e)\bigr]^2  d \mu \\
&=& \sup_n \int  \|(\nabla^2 \psi_n)^{-1/2} \nabla^2 (\psi_n)_e (\nabla^2 \psi_n)^{-1/2} \|_{HS}^2  \  d \mu \geq \sup_n \int \frac{\| \nabla^2 (\psi_n)_e\|_{HS}^2}{\| \nabla^2 \psi_n\|_{HS}^2}  d \mu.
\end{eqnarray*}
By the reverse H\"older inequality we get for all $0 < \varepsilon <1$,
$$
 \int \frac{\| \nabla^2 (\psi_n)_e\|_{HS}^2}{\| \nabla^2 \psi_n\|_{HS}^2}  d \mu \geq \left(\int \| \nabla^2 (\psi_n)_e\|_{HS}^{2-\varepsilon} \right)^ \frac{2}{2-\varepsilon}
 \left(\int \| \nabla^2 \psi_n \|_{HS}^\frac{2(2-\varepsilon)}{\varepsilon} \right)^{- \frac{\varepsilon}{2-\varepsilon}}.
 $$
Together with (\ref{sobol-bound}) 
we obtain the following bound on the third derivatives of $\psi_n$, 
$$
\sup_n \int \| \nabla^2 (\psi_n)_e \|^{2-\varepsilon} d \mu < \infty.
$$
Since $e^{-V}$ is locally stricktly positive inside of $\{V <\infty\}$, we get, in particular, that for every compact set $K \subset \{V<\infty\}$
$$
\sup_n \int_{K} \| \nabla^2 (\psi_n)_e \|^{2-\varepsilon} d x< \infty.
$$
Applying the Rellich--Kondrashov embedding theorem and passing to a subsequence (denoted again by $\{\psi_n\}$),    one can assume that
all the second derivatives $\partial^2_{x_i x_j} \psi_n$  converge almost everywhere.  Applying the same arguments and using the bounds
(\ref{sobol-bound}) one can assume, in addition, that $\psi_n, \partial_{x_i} \psi_n$ have limits almost everywhere  (hence in every $L^p(\mu)$
and $L^p_{loc}(K)$ for all $p>1$ and all compact $K \subset \{V < \infty\}$). In addition, one can assume that the third derivatives 
$(\psi_n)_{x_i x_j x_k}$ converge weakly in $L^{2-\varepsilon}(\mu)$ and $L^{2-\varepsilon}_{loc}(K)$ for every $0<\varepsilon<1$. 
Let us denote the limit of $\psi_n$ by $\psi$. Clearly, $\psi$ is a convex function.  Let us show that 
$$\partial_{x_i} \psi_n \to  \partial_{x_i} \psi.$$
Denote the limit of $\partial_{x_i} \psi_n $ by $f$.
Choose a smooth function $\xi$ with compact support $K \subset \{V < \infty\} $. Using convergence in $L^p_{loc}(dx)$ one gets
$$
\int f \xi dx = \lim_n \int \partial_{x_i} \psi_n \xi dx = - \lim_n \int  \psi_n \partial_{x_i} \xi dx = - \int  \psi \partial_{x_i} \xi dx .
$$
Hence $f$ is the Sobolev partial derivative of $\psi$. Since $\psi$ is convex, it coincides almost everywhere with  $\partial_{x_i} \psi$
in the classical sense.
In the same way we prove that $\psi$ admits second Sobolev derivatives and
$$\partial_{x_i x_j} \psi_n \to  \partial_{x_i x_j} \psi.$$
Finally, using weak convergence of the third derivatives in $L^{1}_{loc}(K)$, we show that $\psi$ admits third order Sobolev derivatives, which are the weak
limits of the corresponding third derivatives of $\psi_n$.
\par
\noindent
Let us pass to the limit in (\ref{vee}). By Fatou's lemma, 
$$
\liminf_n \int \langle \nabla^2 W_n(\nabla \psi_n) \nabla (\psi_n)_e, \nabla (\psi_n)_e \rangle d \mu 
\ge \int \langle\nabla^2 W(\nabla \psi) \nabla \psi_e, \nabla \psi_e \rangle d \mu.
$$
Next we note that
$$
{\rm Tr}\bigl[ (\nabla^2 \psi_n)^{-1}(\nabla^2 (\psi_n)_e)\bigr]^2 = \| A_n\|^2_{HS},
$$
where $\|\cdot\|_{HS}$ is the Hilbert--Schmidt norm and  
$$
A_n= (\nabla^2 \psi_n)^{-1/2} \nabla^2 (\psi_n)_e (\nabla^2 \psi_n)^{-1/2}.
$$
The space of matrix-valued functions $M(x)$ with the norm $\Bigl( \int \| M\|^2_{HS} d\mu \Bigr)^{\frac{1}{2}}$ is a Hilbert space.
By (\ref{Ungleichung}), $\sup_n \int \| A_n\|^2_{HS} d\mu \leq  \int_{\{V < \infty\}} V_{ee} d \mu$, and therefore $\{A_n: n \in \mathbb{N}\}$ is relatively weakly compact in 
this Hilbert space.
Hence there exists  a subsequence of $\{A_n\}$, which we denote again by $\{A_n\}$,  that converges weakly $A_n \to A$  in 
the space of matrix-valued functions. 
Take a matrix valued mapping $M(x)$ such that $\Bigl( \int \| M\|^2_{HS} d\mu \Bigr)^{\frac{1}{2}} <\infty$.
One has
$$
\lim_n \int {\rm Tr}\Bigl( (\nabla^2 \psi_n)^{1/2} M(x) (\nabla^2 \psi_n)^{1/2} A_n \Bigr) d \mu = \int {\rm Tr}\Bigl( (\nabla^2 \psi)^{1/2} M(x) (\nabla^2 \psi)^{1/2} A\Bigr) d \mu.
$$
On the other hand
\begin{align*}
\lim_n & \int {\rm Tr}\Bigl( (\nabla^2 \psi_n)^{1/2} M(x) (\nabla^2 \psi_n)^{1/2} A_n \Bigr) d \mu =\lim_n \int {\rm Tr}\Bigl(  M(x)  \nabla^2 (\psi_n)_e  \Bigr) d \mu 
\\& =\int {\rm Tr}\Bigl(  M(x)  \nabla^2 \psi_e  \Bigr) d \mu .
\end{align*}
Hence
$$
\int {\rm Tr}\Bigl( (\nabla^2 \psi)^{1/2} M(x) (\nabla^2 \psi)^{1/2} A\Bigr) d \mu =\int {\rm Tr}\Bigl(  M(x)  \nabla^2 \psi_e  \Bigr) d \mu. 
 $$
 This implies
$$
A=(\nabla^2 \psi)^{-1/2} \nabla^2 \psi_e   (\nabla^2 \psi)^{-1/2}. 
$$
By the properties of the weak convergence and Fatou's lemma, 
$$
\liminf_n \int \|A_n\|^2_{HS} d\mu \ge \int \|A\|^2_{HS} d\mu.
$$
Passing to the limit in (\ref{vee}),  we get (\ref{l2-ineq}).
\end{proof}
\vskip 3mm
\noindent

\begin{proof}[Proof of Theorem \ref{thm01}]

Let us show that $\psi$ is a smooth function on  $\Omega_\psi =  \mathrm{int} \, (\{ \psi < \infty\})$.
Since 
$$
\int |\nabla \psi|^p d \mu = \int |x|^p d \nu < \infty
$$
for every $p>0$, we get with (\ref{d2p}) that $\int |\psi_{x_i x_i}|^p d \mu < \infty$ for every $p>0$. In particular, by the Sobolev embedding theorem
$\psi$ is locally $C^{1,\alpha}$ on  $\Omega_\psi$. 
Repeating arguments from Proposition \ref{l2ineq}, using continuity
of $\nabla \psi$ and the fact that $\psi$ and  $\psi^*$ are locally 
H{\"o}lder, we obtain that $\psi$ is $C^{2,\alpha}$, $0 < \alpha <1$, by Theorem 4.14  from \cite{Villani}.
Applying higher order regularity theory (see e.g., Remark 4.15 of \cite{Villani}),  we get by bootstrapping arguments  that $\psi$ is  $C^{\infty}$ on $\Omega_\psi$.
\par
\noindent
Thus we are in position to apply Proposition \ref{l2ineq}. In our particular case it reads as 
$$
\int_{\Omega_\psi} \psi_{x_i x_i} d \mu \ge \int \langle\nabla^2 \psi^* (\nabla \psi) \nabla \psi_{x_i}, \nabla \psi_{x_i} \rangle d \mu
+\int {\rm Tr} \bigl[ (\nabla^2 \psi)^{-1} \nabla^2 \psi_{x_i} \bigr]^2 d \mu.
$$
Note that 
$$
\langle \nabla^2 \psi^*  (\nabla \psi) \nabla \psi_{x_i}, \nabla \psi_{x_i} \rangle = \psi_{x_i x_i}.
$$
Here we use that $\nabla^2 \psi^* (\nabla \psi) = \nabla^2 \psi^{-1}$.
\newline
Thus we get $\int {\rm Tr} \bigl[ (\nabla^2 \psi)^{-1} \nabla^2 \psi_{x_i} \bigr]^2 d \mu =0$
and $\nabla^2 \psi_{x_i}  = 0$ on  $\Omega_\psi$.
Hence there exists  a positive matrix $A$, a vector $b$, and an absolute constant $c$ such that
$$
\psi(x) =\frac{1}{2} \langle Ax, x\rangle + \langle b, x \rangle + c,
$$
for every $x$ satisfying $\psi(x) < \infty$.
\newline
Next we note that the push forward measure of 
$$
e^{-\psi} dx = e^{-(\frac{1}{2} \langle Ax, x\rangle + \langle b, x \rangle + c)} dx
$$
under $y= Ax + b = \nabla \psi(x)$ is
$$
C e^{-(\frac{1}{2} \langle A^{-1} (y-b), (y-b)\rangle + \langle b, A^{-1}(y-b) \rangle + c)} dy.
$$
Since 
\begin{align*}
\psi^*(y) = \frac{1}{2} \langle A^{-1} (y-b), (y-b)\rangle - c,
\end{align*}
we immediately obtain that the push forward of $e^{-\psi}$ under $\nabla \psi$ coincides with $C e^{-\psi^*}$
if and only if $b=0$. Thus we get that
$$
\mu = \frac{1}{Z} \  I_{E} \  e^{-\frac{1}{2} \langle Ax, x\rangle }
$$
for some convex set $E$. We conclude the proof with the observation that $E= \mathbb{R}^n$,
as otherwise $\mu$ has no logarithmic derivative.
\end{proof}

\vskip 1cm

\section{ Pinsker type  inequalities}\label{Section-Pinsker}

The original Pinsker  inequality compares two important concepts from information theory, 
the total Variation Distance $V = V(P,Q)$ and Kullback-Leibler divergence $D= D_{KL} (P||Q) $ (see Section \ref{section-Background on $f$-divergence} for the definitions). 
Comparing these two notions has many advantages, such as transferring results from information theory to probability theory or vice versa (see, e.g. \cite{Fedotov-Harremoes-Topsoe2003/1, Fedotov-Harremoes-Topsoe2003/2}).
 Pinsker \cite{Pinsker1960} obtained the following inequality  
\begin{equation} \label{PinskerIneq}
 D \geq \frac{1}{2} V^2 .
\end{equation}
The best constant, $\frac{1}{2}$, is due, 
independently to Csisz\'ar \cite{Csiszar1967}, Kemperman \cite{Kemperman1969} and Kullback \cite{Kullback1967, Kullback1970}.
For  applications of Pinsker's inequality, see e.g. \cite{Barron1986}, \cite{Csiszar1984}, \cite{Topsoe1979}.
\par
\noindent 
The concept of $f$-divergence  is a generalization of Kullback-Leibler divergence.  Thus one wonders whether a Pinsker type inequality holds also for $f$-divergences. This question was  answered by G. Gilardoni \cite{Gilardoni2010} and,   with  different kind of result, by M. Reid and R. Williamson \cite{Reid-Williamson2009}. 
\par
\noindent
G. Gilardoni   \cite{Gilardoni2010} established the following Pinsker type inequality for $f$-divergences.
\begin{theo}\cite{Gilardoni2010}\label{Gilardoni}
Let $f: (0, \infty) \rightarrow \R$ be convex  and $f(1) =0$. Suppose that the convex function $f$ is differentiable up to order 3 at 1 with $f''(1) >0$ and the following inequality holds
\begin{eqnarray}\label{condition-f-div1}
\bigg( f(u) - f'(1) (u-1) \bigg) \left( 1 - \frac{f'''(1)}{3 f''(1)}(u-1) \right) \ \geq \ \frac{f''(1)}{2} (u-1)^2 .
\end{eqnarray}
Then
$$D_f (P,Q) \geq \frac{f''(1)}{2} V^2.$$ 
The constant $\frac{f''(1)}{2}$ is best possible.
If $f$ is concave, the inequalities are reversed.
\end{theo}
\par
\noindent
In this section, we will consider  probability densities. We put 
\begin{equation}\label{Q,P,Normalized}
q_\varphi= \frac{\varphi}{\int_{X_{\psi}} \hskip -1mm \varphi}  \hskip 4mm \text{and} \hskip 4mm   p_\varphi=  \frac{ e^{\frac{\langle\grad\varphi, x \rangle}{\varphi}} \mbox{det} \left[\nabla^2 \left(-\ln \varphi \right)\right] }{\varphi \int_{X_{\psi^*}} \hskip -1mm \varphi^{\circ} }
\end{equation}
We use the expressions (\ref{Q,P,Normalized}) to define  the normalized $f$-divergences for log concave  functions \cite{CaglarWerner}. 
\begin{defn} \label{D2Normalized}
Let $f: (0, \infty) \rightarrow \mathbb{R}$ be  a  convex or concave function and let $ \varphi:\R^{n}\rightarrow [0, \infty)$ be a log concave function.
Then the normalized $f$-divergence $\D_f (P\varphi, Q_\varphi)$ of $\varphi$ is 
\begin{equation}\label{div-Logconcave1_Normalized}
\D_f(\varphi) = \D_f(P_\varphi, Q_\varphi)=
\int_{X_{\psi}}   \frac{\varphi}{ \int_{X_{\psi}} \hskip -1mm\varphi}  \ f  \left(
\frac{e^{\frac{\langle\grad\varphi, x \rangle}{\varphi}} }{\varphi ^2} \ \frac{\int_{X_{\psi}} \hskip -1mm \varphi}{\int_{X_{\psi^*}} \hskip -1mm \varphi^\circ} \mbox{det} \left[ \nabla^2 \left( -\ln \varphi \right)\right] \right)dx.
\end{equation}
\end{defn}
\vskip 2mm
\noindent
From the result of G. Gilardoni \cite{Gilardoni2010}, we obtain, under additional conditions (see Section \ref{Gillardoni}),  an  information-theoretic  inequality for log concave functions.
\vskip 3mm
\noindent
\subsection{Pinsker inequalities for log concave functions}\label{Gillardoni}
The following entropy inequality for log concave functions follows as  an immediate consequence of Theorem \ref{Gilardoni} with the densities  (\ref{Q,P,Normalized}) when, for a convex function $f$,  (\ref{condition-f-div1}) is satisfied:
\begin{eqnarray}\label{main2}
\D_f(\varphi) \ \geq \ \frac{f''(1)}{2}  
\left( \int_{X_\psi} \big| \frac{ e^{\frac{\langle\grad\varphi, x \rangle}{\varphi}} \mbox{det} \left[\nabla^2 \left(-\ln \varphi \right)\right] }{\varphi \int_{X_{\psi^*}} \hskip -1mm \varphi^{\circ} } \ - \ \frac{\varphi}{\int_{X_{\psi}} \hskip-1mm \varphi} \big| dx \right)^2 ,
\end{eqnarray}
with  equality  if $\varphi (x) = e^{- \frac{1}{2} {\langle A x, x \rangle}}$ where $A$ is an $n \times n$  positive definite  matrix. 
If $f$ is concave, the inequality is reversed.
\par
\noindent
This inequality  is stronger than the inequality from \cite{CaglarWerner} stated in Theorem \ref{thm00}.   Indeed,  Theorem \ref{thm00}  says that 
$D_f (\varphi) \geq f(1)$ and $f(1)=0$  for probability densities. 
Example \ref{ConcreteExample} shows  that the right hand side of (\ref{main2}) however  is not always 0.
\vskip 2mm
\noindent
We want to concentrate on the case when $f(t) = -\ln t$.  Then the assumptions of Gilardoni 's theorem hold and we get the following corollary.
\par
\noindent
Recall also that  $\operatorname{Ent}(\varphi) =\int_{\Omega_\psi} \varphi \ln \left( \varphi \right) dx  \ -  \ \int_{\Omega_\psi} \varphi \ln \left( \int_{\Omega_\psi} \varphi \right) dx $.
By $g$, we denote the Gaussian which has entropy $\operatorname{Ent}(g) = - \frac{n}{2}  \ln (2 \pi e)$.  We also use that for functions  $ \varphi \in C^2$, we have that $X_\psi = \Omega_\psi$ and that
$ n \int_{\Omega_\psi} \varphi = -  \int_{\Omega_\psi}  \hskip -1mm \langle\grad\varphi, x \rangle dx $.

\begin{cor}\label{PinskerIneqFunctional}
Let $ \psi:\R^{n}\rightarrow \mathbb{R}$ be a
convex function such that $\varphi = e^{-\psi}$ is a probability density. Then
\begin{eqnarray}\label{eqn0:Korollar1}  
 \hskip -5mm  \int_{X_\psi}   \ \ln \bigg(\det \left( \nabla^2  \psi \right)\bigg)  e^{- \psi(x) } dx  &\leq& 2\left[ \operatorname{Ent}(\varphi) - \operatorname{Ent}(g) \right] 
+ \ln \left( \frac{\int_{X_{\psi^*}} e^{- \psi^*} }{(2  \pi)^n}\right)  \nonumber \\
& - &\frac{1}{2}
\left(
 \int_{X_\psi}  \left| \frac{ e^{\psi - \langle \grad \psi, x \rangle} \mbox{det} \left (\nabla^2 \psi \right) } { \int_{X_{\psi^*}} e^{-\psi^*} } 
\ - \   e^{-\psi}  \right| dx 
\right)^2, 
\end{eqnarray}
with equality if $\varphi (x) = e^{- \pi {\langle A x, x \rangle}}$ where $A$ is an $n \times n$  positive definite  matrix with $\det (A) =1$.
\end{cor}
\vskip 2mm
\noindent
{\bf Remarks.} 
\par
\noindent
(i)  Note that 
$$
 \int_{X_\psi}  \left| \frac{ e^{\psi - \langle \grad \psi, x \rangle} \mbox{det} \left (\nabla^2 \psi \right) } { \int_{X_{\psi^*}} e^{-\psi^*} } 
\ - \   e^{-\psi}  \right| dx  \geq  \left| \int_{X_\psi} \frac{ e^{\psi - \langle \grad \psi, x \rangle} \mbox{det} \left (\nabla^2 \psi \right) } { \int_{X_{\psi^*}} e^{-\psi^*} } 
\ - \   e^{-\psi}  \ dx \right|  =0.
$$
The last equality holds as $\varphi$ is a probability density and by (\ref{transformation}). 
Therefore inequality (\ref{eqn0:Korollar1}) implies the following inequality, which was proved in   \cite{CaglarWerner}, 
\begin{eqnarray}\label{Gleichung2}
 \int_{X_\psi}   \ \ln \bigg(\det \left( \nabla^2  \psi \right)\bigg)  e^{- \psi(x) } dx  &\leq& 2\left[ \operatorname{Ent}(\varphi) - \operatorname{Ent}(g) \right] 
+ \ln \left( \frac{\int_{X_\psi^*} e^{- \psi^*} }{(2  \pi)^n}\right).
\end{eqnarray}
\par
\noindent
(ii) The functional Blaschke Santalo inequality \cite{ArtKlarMil, KBallthesis, FradeliziMeyer2007,Lehec2009} implies that,  for a probability density $\varphi$,   $\left( \frac{\int_{X_\psi^*} e^{- \psi^*} }{(2  \pi)^n}\right) \leq 1$.
Therefore inequality (\ref{Gleichung2}) implies 
\begin{eqnarray}\label{Gleichung3}
 \int_{X_\psi}   \ \ln \bigg(\det \left( \nabla^2  \psi \right)\bigg)  e^{- \psi(x) } dx  &\leq& 2\left[ \operatorname{Ent}(\varphi) - \operatorname{Ent}(g) \right] ,
\end{eqnarray}
which was proved in  \cite{ArtKlarSchuWer}. 
Thus inequality (\ref{eqn0:Korollar1}) is the strongest of those entropy inequalities. Indeed, the next example shows that the additional term on the right hand side of (\ref{eqn0:Korollar1})
is not equal to $0$ in general.
\vskip 3mm
\noindent
\begin{exam}\label{ConcreteExample}
Let $p>1$ and $ \varphi:\mathbb R^{n}\rightarrow\mathbb R$ be given by $ \varphi(x)=  \frac{1}{A} \cdot e^{- \frac{1}{p} \sum_{i=1}^n |x_i|^p}$, where $A= \int e^{- \frac{1}{p} \sum_{i=1}^n |x_i|^p} = 2^n \big(\Gamma (\frac{1}{p}) \ p^{\frac{1-p}{p}}\big)^n $. 
Then $\varphi^\circ = A \cdot e^{- \frac{1}{q} \sum_{i=1}^n |x_i|^q}$, where $\frac{1}{p} + \frac{1}{q} = 1$.  Moreover,
$$
\int \varphi^\circ = \int  e^{-\psi^*} =  \int  A e^{- \frac{1}{q} \sum_{i=1}^n |x_i|^q }= A \cdot 2^n  \left(\Gamma \left(\frac{1}{q}\right) \ q^{\frac{1-q}{q}}\right)^n ,
$$
which is just $(2 \pi)^n$ if $p=q=2$.  If $p$ is not $2$, then it is not necessarily $2\pi$.
And
\begin{eqnarray*}
 && \int  \left| \frac{ e^{\psi - \langle \grad \psi, x \rangle} \mbox{det} \left (\nabla^2 \psi \right) } { \int  e^{-\psi^*} } 
\ - \   e^{-\psi}  \right| dx  \nonumber \\
&& \hskip 3cm =  \int  \left|  \frac{  (p-1)^n \prod_{i=1}^{n} x_i^{p-2} }{2^n \big(\Gamma(\frac{1}{q}) \ q^{\frac{1-q}{q}} \big)^n } \  e^{ \frac{1-p}{p} \sum_{i=1}^n |x_i|^p}  \ - \  A^{-1} e^{- \frac{1}{p} \sum_{i=1}^n |x_i|^p}  \right| dx,
\end{eqnarray*}
which is only equal to $0$,  if $p=q=2$.  If $p$ is not $2$, then it is not necessarily $0$.
\end{exam}
\vskip 3mm

\section{Entropy inequalities}\label{SectionResults1}
The following theorem, which provides  bounds for $f$ divergence,  was proved by  S. Dragomir  \cite{DRAGOMIR_1}.
Dragomir proved the  theorem in the discrete case, that is $p, q \in \R^n_+$, $\nu$ is the counting measure on $\{1, \cdots,n\}$, $P=p  \nu_n$ and $Q= q \nu_n$.
Then
$$
D_f ( P , Q) = \sum_{i=1}^n f\left(\frac{p_i}{q_i} \right ) q_i .
$$
\begin{theo}\label{thmDragomirOriginal}\cite{DRAGOMIR_1}
Suppose that the convex function   $f : (0, \infty) \rightarrow \R $  is differentiable. Then we have for for all discrete densities $p, q \in \R^n_+$, 
\begin{eqnarray}\label{StatementthmDragomirOriginal}
f'(1) \left( P_n - Q_n \right)  \ \leq  \ D_f ( P , Q) - f(1) Q_n  \
\leq \  D_{f'} \left( \frac{ P^2 }{Q } , P \right) - D_{f'} (P , Q),
 \end{eqnarray}
 where $f' : (0, \infty) \rightarrow \R $ is the derivative of $f$ and $P_n = \sum_{i=1}^{n} p_i >0$,  $Q_n = \sum_{i=1}^{n} q_i >0$.
If $f$ is concave, the inequality is reversed.  
If $f$ is strictly convex (respectively  strictly concave) and $p_i, q_i >0$ for all $1 \leq i \leq n$,  then equality holds iff $p = q$.
\end{theo}
\vskip 3mm
\noindent
The result  still holds,
 if the discrete densities are replaced by   $\mu$-a.e. positive  general  densities $p$ and $q$ defined on a measure space $(X,  \mu)$.
The proof is  the  same as the one  given by  by Dragomir in the discrete case  \cite{DRAGOMIR_1}.
We include it for completeness.
We state it when $f : (0, \infty) \rightarrow \R $  is convex. If $f$ is concave, the inequalities are reversed.
We use the notation $I_P=\int_X p d\mu$ and $I_Q=\int_X q d\mu$.
\vskip 3mm
\begin{theo} \label{thmDragomir}
Let $f : (0, \infty) \rightarrow \R $  be  differentiable and convex. Let $p$ and $q$ be  $\mu$-a.e. positive  densities  on $X$ such that $I_p$ and $I_q$ are finite.
Then 
\begin{equation}\label{generalDragomir}
 I_Q    \  f\left(\frac{ I_P}{ I_Q }  \right)  \leq   D_f ( P , Q)   \leq  f(1) I_Q  +  D_{f'} \left( \frac{ P^2 }{Q } , P \right) - D_{f'} (P , Q)
\end{equation}
 If $f$ is linear, equality holds on both sides 
and we get
\begin{equation*}
D_f ( P , Q)  =   f(1) I_Q  \ + f'(1) \left( I_Q - I_P \right) .
\end{equation*}
If $f$ is strictly convex,  equality holds on the left  inequality if and only if $p= c \  q$ $\mu$-a.e where  $c>0$ is a constant and equality holds on the right inequality if and only if 
$p=  q$  $\mu$-a.e .
\end{theo}
\par
\noindent
\begin{proof}[Proof of Theorem \ref{thmDragomir}]
By Jensen's inequality, 
\begin{eqnarray*}
D_f ( P , Q)  &=& \int_X f\left(\frac{p}{q}\right)  q d \mu  = I_Q  \   \int_X f \left(\frac{p}{q}\right)  \frac{q}{I_Q} d \mu  \ \geq \  I_Q  \   f  \left( \int_X   \frac{p}{I_Q} d \mu  \right) \\
&=& I_Q  \   f \left(   \frac{I_P}{I_Q} \right), 
\end{eqnarray*}
which proves the left inequality of (\ref{generalDragomir}).
\par
\noindent
It is easy to see that equality holds on both sides of (\ref{generalDragomir}) if $f$ is linear.  If $f$ is strictly convex, 
equality holds in the left inequality of   (\ref{generalDragomir}) if and only if equality holds  in Jensen's inequality which happens if and only if  $p= c \ q$ $\mu$-a.e where  $c>0$ is a constant. 
\par
\noindent
Since $f$ is differentiable and convex, we have for all $s, t \in (0, \infty)$, 
\begin{equation}\label{concavity}
 f' (t) (t-s ) \geq f(t) - f(s) \geq f'(s) (t-s).
\end{equation}
Let  $x \in X$ be such that $q(x) >0$ and  let $ t = \frac{p(x)}{q(x)}$. As $q>0$ $\mu$-a.e., it is enough to consider only such $x$.  Let $ s=1$. By  inequality (\ref{concavity}),
$$  f' \left(\frac{p(x)}{q(x)}\right) \left(\frac{p(x)}{q(x)}- 1 \right) \geq f \left(\frac{p(x)}{q(x)}\right) - f(1) 
.$$
We multiply both sides by $q(x)$ and integrate
$$  \int_X (p(x) -q(x) ) f' \left(\frac{p(x)}{q(x) } \right) d \mu(x)  \geq  D_f(P, Q) - f(1) \int_X q(x) d \mu(x) 
.$$
Since
$$ \int_X (p(x) -q(x) ) f' \left(\frac{p(x)}{q(x) } \right) d \mu(x)  =    D_{f'} \left( \frac{ P^2 }{Q } , P \right) - D_{f'} (P , Q), $$
we obtain the desired result.
\par
\noindent
Equality holds in (\ref{concavity}) for  a strictly convex   function $f$ iff $s=t$. Therefore, if $f$ is strictly  convex,  equality holds  on the right inequality of (\ref{generalDragomir}),
iff $p=q$  $\mu$-a.e .
\end{proof}
\par
\noindent
{\bf Remark.}  Under  the assumptions of Theorem \ref{thmDragomir}, we have that 
\par
\noindent
(i)  When $p$ and $q$ are probability densities, then
\begin{equation}\label{ProbaDragomir}
f(1)   \ \leq  \ D_f ( P , Q) 
\  \leq  f(1)   \ +  D_{f'} \left( \frac{ P^2 }{Q } , P \right) - D_{f'} (P , Q) 
\end{equation}
and  that  $D_f ( P , Q)  =   f(1)$, if $f$ is linear. When  $f$ is strictly convex,  equality holds in both  inequalities if and only if $p=  q$ $\mu$-a.e.
\par
\noindent
(ii) 
If we let $t= \frac{I_P}{I_Q}$ and $s=1$ in (\ref{concavity}), then  
$$
I_Q  \   f \left(   \frac{I_P}{I_Q} \right) \geq  f(1) I_Q  \ + f'(1) \left( I_P - I_Q \right), 
$$
which, together with the  upper bound of (\ref{generalDragomir}), leads to inequalities corresponding to (\ref{StatementthmDragomirOriginal}), 
\begin {equation}
 f'(1) \left( I_P - I_Q \right) \ \leq  \ D_f ( P , Q)  - f(1) I_Q 
\leq  D_{f'} \left( \frac{ P^2 }{Q } , P \right) - D_{f'} (P , Q).
\end{equation}
When $f$ is linear, equality holds in both inequalities. 
When $f$ is  strictly convex,  equality holds in both inequalities if and only  if $p =  q$ $\mu$-a.e.
\vskip 3mm
\noindent
With the identity (\ref{transformation}), the following corollary is immediate from Theorem \ref{thmDragomir}. 
\begin{cor}\label{corDragomir} 
Let $ \varphi:\R^{n}\rightarrow [0, \infty)$ be a
log-concave
function and $f : (0, \infty) \rightarrow \R $  be  differentiable and convex. Then
\begin{eqnarray}\label{StatementthmDragomir}
&& \hskip -10mm f\left( \frac{\int_{X_{\psi^*}} \varphi^{\circ} }{\int_{X_{\psi}} \varphi}  \right)  \int_{X_{\psi}}  \varphi   \ \leq  \nonumber \\
&& \ D_f ( P_{\varphi} , Q_{\varphi}) \
\leq \ f(1) \int_{X_{\psi}} \varphi  +  D_{f'} \left( \frac{ P_{\varphi}^2 }{Q_{\varphi} } , P_{\varphi} \right) - D_{f'} (P_{\varphi} , Q_{\varphi}).
 \end{eqnarray}
 \par
\noindent
If $f$ is concave, the inequality  is reversed. If $f$ is linear, equality holds in both inequalities.
\par
\noindent
If $f$ is strictly convex or strictly concave,   then equality holds on the left hand side if  $\varphi (x) = c \ e^{- \frac{1}{2} {\langle A x, x \rangle}}$ where $A$ is an $n \times n$  positive definite  matrix and $c>0$ is an absolute constant.
\par
\noindent
 If   $\varphi$ is in addition $C^2$, then equality holds on the left hand side iff  $\varphi (x) = c \ e^{- \frac{1}{2} {\langle A x, x \rangle}}$ where $A$ is an $n \times n$  positive definite  matrix and $c>0$ is an absolute constant.
 \par
\noindent
If $f$ is strictly convex or strictly concave,   then equality holds on the right hand side if  $\varphi (x) =  e^{- \frac{1}{2} {\langle A x, x \rangle}}$ where $A$ is an $n \times n$  positive definite  matrix with $\det (A) =1$.
\par
\noindent
If  $\varphi$ is in addition $C^2$, then equality holds on the right hand side iff  $\varphi (x) = e^{- \frac{1}{2} {\langle A x, x \rangle}}$ where $A$ is an $n \times n$  positive definite  matrix with $\det (A) =1$.
\end{cor}
\par
\noindent
\textbf{Remark.}
 Inequality (\ref{StatementthmDragomir}) is  invariant under self adjoint SL(n) maps.
This follows as both, 
$D_f ( P_{\varphi} , Q_{\varphi}) $  and  $D_{f'} \left( \frac{ P_{\varphi}^2 }{Q_{\varphi} } , P_{\varphi} \right)$ are  invariant under self adjoint SL(n)maps, with possibly different degree of homogeneity.
For $D_f ( P_{\varphi} , Q_{\varphi}) $ this was proved  in \cite{CaglarWerner}.
For 
$D_{f'} \left( \frac{ P_{\varphi}^2 }{Q_{\varphi} } , P_{\varphi} \right)$, it is shown similarly.
\vskip 2mm
\noindent
\begin{proof}[Proof of Corollary  \ref{corDragomir}]
With the identity (\ref{transformation}) and the densities 
\begin{equation*}\label{Q,P}
q_\varphi=\varphi  \hskip 4mm \text{and} \hskip 4mm   p_\varphi= \varphi^{-1}  e^{\frac{\langle\grad\varphi, x \rangle}{\varphi}} \mbox{det} \left[ \nabla^2  \left(-\ln \varphi \right)\right],
\end{equation*}
the inequalities of the corollary follow immediately from Theorem \ref{thmDragomir}.
\par
\noindent
Let  $A$ be a positive definite $n \times n$ matrix, $c>0$ a constant and $\varphi(x) = c e^{- \frac{1}{2} \langle Ax, x \rangle}$. Then
\begin{equation*}\label{exponential2} 
D_f(P_\varphi, Q_\varphi ) \ = \ f \bigg( \frac{ \det (A)}{c^2} \bigg) \frac{ c (2 \pi )^{n/2}}{\sqrt{\det (A)}} .
\end{equation*}
Therefore it is easy to see that we have equality on the left hand side if  $\varphi (x) = c \ e^{- \frac{1}{2} {\langle A x, x \rangle}}$
and on the right hand side if  $\varphi (x) =  e^{- \frac{1}{2} {\langle A x, x \rangle}}$ where $A$ is an $n \times n$  positive definite  matrix with $\det (A) =1$.
\par
\noindent
By Theorem \ref{thmDragomir},  equality holds on the right inequality iff  $p_{\varphi} = q_{\varphi}$ a.e. and on the left iff  $p_{\varphi} =c\  q_{\varphi}$ a.e. where $c>0$ is a constant.
For functions $\varphi (x) = c\ e^{- \frac{1}{2} {\langle A x, x \rangle}}$, where $A$ is an $n \times n$  positive definite  matrix with $\det (A) =1$, it is easy to check that $p_{\varphi} = q_{\varphi}$.
The equation  $p_{\varphi} = c \ q_{\varphi}$ a.e., is equivalent to the equation
\begin{equation*}
\det(\nabla^2\psi(x))=c\, e^{-2\psi(x)+\langle \nabla \psi(x),x\rangle},\qquad  \mbox{a.e.}\;  x\in  {\R}^n .\label{eql2}
\end{equation*}
Then,  if $\varphi $ is $C^2$, Theorem \ref {thm01} and the remarks before it,   finish the proof of the corollary.
\end{proof}
\vskip 2mm
\noindent
Now we consider special cases of  Corollary \ref{corDragomir}.
\par
\noindent
If we let $f (t) = - \ln t$ in the previous corollary, we obtain the following affine invariant entropy inequalities which give upper and lower bounds for the
relative entropy in terms of the functional affine surface areas. The proof follows immediately from Corollary \ref{corDragomir}.
\begin{cor}
Let $ \varphi:\R^{n}\rightarrow [0, \infty)$ be a
log-concave
function. 
Then
\begin{eqnarray}\label{corrr2}
\ln \bigg( \frac{as_0 (\varphi)}{as_1 (\varphi)} \bigg)  as_0 (\varphi)  \ \leq D_{KL} (Q_{\varphi}||P_{\varphi}) \leq \ as_{-1} (\varphi) - as_0 (\varphi). 
\end{eqnarray}
\par
\noindent
Equality holds on the left hand side if  $\varphi (x) = c \ e^{- \frac{1}{2} {\langle A x, x \rangle}}$ where $A$ is an $n \times n$  positive definite  matrix and $c>0$ is an absolute constant 
and equality holds on the right hand side if  $\varphi (x) =  e^{- \frac{1}{2} {\langle A x, x \rangle}}$ where $A$ is an $n \times n$  positive definite  matrix with $\det (A) =1$, 
\par
\noindent
If $\varphi$ is in addition $C^2$, then equality holds on the left hand side iff  $\varphi (x) = c \ e^{- \frac{1}{2} {\langle A x, x \rangle}}$ where $A$ is an $n \times n$  positive definite  matrix and $c>0$ is an absolute constant. And equality holds on the right hand side iff  $\varphi (x) = e^{- \frac{1}{2} {\langle A x, x \rangle}}$ where $A$ is an $n \times n$  positive definite  matrix with $\det (A) =1$.
\end{cor}
 \par
\noindent
\textbf{Remarks.} 1. If, in the previous corollary,   $\varphi$ is a probability density, then the inequalities become
\begin{eqnarray*}
1-\ln \bigg( as_1 (\varphi) \bigg)    \ \leq  1+ D_{KL} (Q_{\varphi}||P_{\varphi}) \leq \ as_{-1} (\varphi) .
\end{eqnarray*}
\par
\noindent
2. Applying (\ref{concavity}), with $f(t)= \ln t$, to the left inequality of the previous  corollary, we get
\begin{eqnarray}\label{corrr23}
as_0 (\varphi)  - as_1 (\varphi)  \ \leq D_{KL} (Q_{\varphi}||P_{\varphi}) \leq \ as_{-1} (\varphi) - as_0 (\varphi). 
\end{eqnarray}
\vskip 2mm
\noindent
If we let $f(t) = t^{\lambda}$ in inequality (\ref{StatementthmDragomir}),  then we obtain  functional affine isoperimetric  inequalities. 
  \begin{cor}\label{aff3} 
Let $ \varphi:\R^{n}\rightarrow [0, \infty)$ be a log concave function.
\par
(i) If  $\lambda \geq 1$ or $\lambda \leq 0$, 
\begin{eqnarray}\label{results15}
  as_1^{\lambda} (\varphi)   \ as_0^{1 - \lambda} (\varphi) \  \leq  \  as_\lambda(\varphi)   \ \leq \   as_0 (\varphi)  + \lambda \left(   as_\lambda(\varphi)   -  as_{\lambda -1} (\varphi) \right).
\end{eqnarray}
\par
(ii) If  $\lambda \in (0,1)$, 
\begin{eqnarray}\label{results16}
  as_1^{\lambda} (\varphi)   \ as_0^{1 - \lambda} (\varphi) \  \geq  \  as_\lambda(\varphi)   \ \geq \   as_0 (\varphi)  + \lambda \left(   as_\lambda(\varphi)   -  as_{\lambda -1} (\varphi) \right).
\end{eqnarray}
\par
\noindent
Equality holds trivially if $\lambda =1$ or $ \lambda = 0$. 
\par
\noindent
Equality holds on the left hand sides if  $\varphi (x) = c \ e^{- \frac{1}{2} {\langle A x, x \rangle}}$ where $A$ is an $n \times n$  positive definite  matrix and $c>0$ is an absolute constant.
Equality holds on the right hand sides if  $\varphi (x) =  e^{- \frac{1}{2} {\langle A x, x \rangle}}$ where $A$ is an $n \times n$  positive definite  matrix with $\det (A) =1$.
\newline
If   $\varphi$ is in addition $C^2$, then equality holds on the left hand sides iff  $\varphi (x) = c \ e^{- \frac{1}{2} {\langle A x, x \rangle}}$ 
and equality holds on the right hand sides iff  $\varphi (x) = e^{- \frac{1}{2} {\langle A x, x \rangle}}$ 
with $\det (A) =1$.
\end{cor}
\vskip 2mm
\noindent
  \textbf{Remarks.}
1. Applying (\ref{concavity}) to the function  $f(t)= t^\lambda$  for $t=\frac{as_1^{\lambda} (\varphi) }{as_0^{ \lambda} (\varphi)}$ and $s=1$,  we get from the left inequality of the previous  corollary that 
\begin{eqnarray}\label{results11}
\lambda \left(   as_1 (\varphi) -  as_0 (\varphi) \right)  \  \leq  \  as_\lambda(\varphi)  - as_0 ( \varphi ) \ \leq \  \lambda \left(   as_\lambda(\varphi)   -  as_{\lambda -1} (\varphi) \right),
\end{eqnarray}
when $\lambda \geq 1$ or $\lambda \leq 0$ and that 
\begin{eqnarray}\label{results12}
\lambda \left(   as_1 (\varphi) -  as_0 (\varphi) \right)  \  \geq  \  as_\lambda(\varphi)  - as_0 ( \varphi ) \ \geq \  \lambda \left(   as_\lambda(\varphi)   -  as_{\lambda -1} (\varphi) \right),
\end{eqnarray}
when $\lambda \in (0,1)$. 
The equality cases are as in the corollary.
\par
\noindent 
2. The following inequalities for the difference of  functional affine surface areas follow immediately  from (\ref{results11}).
For $\lambda  \geq 1$,
\begin{eqnarray}
   as_1 (\varphi) -  as_0 (\varphi)  \  \leq \    as_\lambda(\varphi)   -  as_{\lambda -1} (\varphi). 
\end{eqnarray}
For $\lambda  \leq 1$, the inequality is reversed.
 \par
 \noindent
 If we let $\lambda = -1$ in inequality (\ref{results11}) and $\lambda =  1/2$ in Corollary \ref{aff3}, then 
\begin{equation*}
 as_{-1} (\varphi) \ \leq \ \frac{  as_0 (\varphi) + 
 as_{-2} (\varphi)  }{2}, \quad  
 as_{0} (\varphi) \ \leq \ \frac{  as_{-1} (\varphi) + 
 as_{1} (\varphi)  }{2}
 \end{equation*}
\begin{equation*}
 as_{0} (\varphi) \ \leq \ \frac{  as_{\frac{1}{2}} (\varphi) + 
 as_{-\frac{1}{2}} (\varphi)  }{2},
\quad
 as_{\frac{1}{2}}(\varphi) \ \leq \   \sqrt{as_{1} (\varphi)  \ 
 as_{0} (\varphi)}.
\end{equation*}
\par
\noindent
Similar results hold for dual function $\varphi^\circ$    by the  duality relation $as_\lambda(\varphi) =  as_{1 - \lambda} (\varphi^\circ)$, which was proved in \cite{CFGLSW}.
\vskip 3mm
\noindent
We have that 
$
 0< \int_{\R^n} \varphi < \infty
 $
by (\ref{positive}) and as $\varphi$ is integrable by assumption. 
Let $\lambda \in (0,1)$ and suppose that $\varphi$ is centered, i.e., $\int_{\R^n} x \varphi(x) dx =0$.  
The functional Blaschke Santal\`o inequality \cite{ArtKlarMil,  KBallthesis, FradeliziMeyer2007, Lehec2009}
says that  for a centered log concave function $\varphi$, 
$$
 \int_{\R^n}\varphi dx \cdot \int_{\R^n} \varphi^\circ  dx\leq (2\pi)^n , 
 $$
 with equality iff $\varphi (x) = C\ e^{- \frac{1}{2} {\langle A x, x \rangle}}$. For $\lambda \in (0,1)$.
We apply this inequality in the left  inequality (\ref{results16}). Also using that $ as_1 (\varphi) = \int_{X_{\psi^*}} \varphi^\circ$, we get
\begin{eqnarray*}  
 as_{\lambda} (\varphi) &\leq& \bigg( \int_{X_{\psi^*}} \varphi^\circ \bigg)^{\lambda}  \bigg( \int_{X_{\psi}} \varphi \bigg)^{1- \lambda}  \leq  \bigg( \int_{\R^n} \varphi^\circ \bigg)^{\lambda} \bigg( \int_{\R^n} \varphi \bigg)^{\lambda} \bigg( \int_{X_{\psi}} \varphi \bigg)^{1- 2\lambda} \\
 &\leq& \left( 2\pi \right)^{n \lambda}  \bigg( \int_{X_{\psi}} \varphi \bigg)^{1- 2\lambda}, 
\end{eqnarray*}
with equality iff $\varphi (x) = c\ e^{- \frac{1}{2} {\langle A x, x \rangle}}$.
\par
 \noindent
 Similarly,  by inequality (\ref{results15}),  we get for $\lambda < 0$,
 $$  
 as_{\lambda} (\varphi) \geq  \left( 2\pi \right)^{n \lambda}  \bigg( \int_{X_{\psi}} \varphi \bigg)^{1- 2\lambda}, 
$$
with equality iff $\varphi (x) = c\ e^{- \frac{1}{2} {\langle A x, x \rangle}}$.
\par
\noindent
A functional version of the inverse Blaschke Santal\`o inequality, due to Fradelizi and Meyer \cite{FM2008}
says that 
$ \int_{\R^n} \varphi   dx  \cdot \int_{\R^n}  \varphi^\circ  dx \geq c^n $, where $c>0$ is a constant.   We use this   for $\lambda >1$ in the left inequality (\ref{results15})
to get that 
$
 as_\lambda(\varphi) \ \geq \ c^{n \lambda}   \left( \int_{X_\psi} \varphi \right)^{ 1 - 2 \lambda}$. 
\vskip 3mm
\noindent
Thus we have proved the following corollary  which was proved in \cite{CFGLSW} by different methods.
\vskip 2mm
 \begin{cor}\label{AffineisopIneq} 
Let $ \varphi:\R^{n}\rightarrow [0, \infty)$ be a log concave  centered function. 
\par
(i)  If $\lambda \in [0,1]$, then 
$ as_\lambda(\varphi) \ \leq \  (2\pi)^{n\lambda} \left( \int_{X_\psi} \varphi \right)^{ 1 - 2 \lambda}$.
\par
(ii)  If  $\lambda  \leq 0$,   then
$ as_\lambda(\varphi) \ \geq \  (2\pi)^{n\lambda} \left( \int_{X_\psi} \varphi \right)^{ 1 - 2 \lambda}$.
\par
(iii) If $\lambda > 1$,  then
$ as_\lambda(\varphi) \ \geq \ c^{n \lambda}   \left( \int_{X_\psi} \varphi \right)^{ 1 - 2 \lambda}$. 
\par
\noindent
Equality holds trivially if  $ \lambda = 0$.

Equality holds in the first two inequalities iff $\varphi (x) = C \ e^{- \frac{1}{2} {\langle A x, x \rangle}}$ where $A$ is an $n \times n$  positive definite  matrix and $C>0$.
 \end{cor}
 \noindent
If $\lambda =1$, the first equality is just the functional Blaschke Santal\`o inequality.
 \vskip 2mm
 \noindent
Another consequence of Corollary \ref{aff3}, along with the duality relation $as_\lambda(\varphi) =  as_{1 - \lambda} (\varphi^\circ)$ (proved in \cite{CFGLSW}),  is the following 
 (functional) Blaschke Santal\'o  type inequalities, which were originally proved in \cite{CFGLSW} by different methods.  
 \begin{cor}\label{FBS}
Let $ \varphi:\R^{n}\rightarrow [0, \infty)$ be a log concave function.   
\par
If $\lambda \in [0,1]$ and $\varphi $ is centered, then 
$as_{\lambda} (\varphi) \ as_{\lambda} (\varphi^\circ) \ \leq  \ (2 \pi)^n $.
\par
If  $\lambda \geq 1$ or $\lambda \leq 0$, then
$
as_{\lambda} (\varphi) \ as_{\lambda} (\varphi^\circ)  \geq \int_{X_{\psi}} \varphi \int_{X_{{\psi^*}}} \varphi^\circ$. 
\par
\noindent
 Equality holds in the first  inequality iff $\varphi (x) = C \ e^{- \frac{1}{2} {\langle A x, x \rangle}}$ where $A$ is an $n \times n$  positive definite  matrix and $C>0$. 
If $\varphi$ is in addition $C^2$, then equality holds in the second inequality iff $\varphi (x) = C \ e^{- \frac{1}{2} {\langle A x, x \rangle}}$ where $A$ is an $n \times n$  positive definite  matrix and $C>0$.
 \end{cor}
 \vskip 2mm
 \noindent
Note that if $\varphi $ is $C^2$, then $X_{\psi} = \Omega_{\psi}$ and $X_{\psi^*} = \Omega_{\psi^*}$. Therefore, when  $\varphi \in C^2$, 
  we have
\begin{eqnarray*}
as_{\lambda} (\varphi) \ as_{\lambda} (\varphi^\circ)  \geq c^n ,
\end{eqnarray*}
for $\lambda \geq 1$ or $\lambda \leq 0$, which follows by the inverse functional  Blaschke Santal\'o inequality.

 \vskip 2mm
 \begin{proof}[Proof of Corollary \ref{FBS}]
 For  $\lambda \in [0,1]$,  Corollary \ref{aff3},  the duality relation $ as_\lambda(\varphi) =  as_{1 - \lambda} (\varphi^\circ)$,  and  the functional  Blaschke Santal\'o inequality yield
\begin{eqnarray*}
as_{\lambda} (\varphi) \ as_{\lambda} (\varphi^\circ)  \leq \int_{X_{\psi}} \varphi \int_{X_{{\psi^*}}} \varphi^\circ \leq  \int_{\R^n} \varphi \int_{\R^n} \varphi^\circ \leq (2 \pi)^n . 
 \end{eqnarray*}
 The second inequality is trivial for  $\lambda \geq 1$ or $\lambda \leq 0$ by Corollary \ref{aff3}, along with the duality relation $ as_\lambda(\varphi) =  as_{1 - \lambda} (\varphi^\circ)$.
 \par
 \noindent
 Equality holds in the first  inequality iff $\varphi (x) = C \ e^{- \frac{1}{2} {\langle A x, x \rangle}}$, where $A$ is an $n \times n$  positive definite  matrix and $C>0$. This follows from the equality characterization of the  functional  Blaschke Santal\'o inequality.  
If $\varphi$ is in addition $C^2$, then equality holds in the second inequality iff $\varphi (x) = C \ e^{- \frac{1}{2} {\langle A x, x \rangle}}$,  where $A$ is an $n \times n$  positive definite  matrix and $C>0$. 
 \end{proof}

\section{Applications}\label{Applications}

In this section we will derive applications to convex bodies. We first recall the notion 
of $f$-divergence for convex bodies. For general information on convex bodies  the books \cite{GardnerBook, SchneiderBook} are excellent sources.
\par
\noindent
 In \cite{Werner2012}, $f$-divergence and their inequalities were introduced for convex bodies.
For details  and special cases we refer to \cite{Werner2012} and give here only the definition. 
\par
\noindent
Let $K$ be a convex body in $\mathbb{R}^n$. We assume throughout that $K$  has center of gravity at $0$.  For $x \in \partial K$, 
the boundary of a sufficiently smooth  convex body $K$, let  $N_K(x)$ denote the outer unit normal to $\partial K$ in $x$ and let $\kappa_K(x)$ be the Gauss curvature in $x$, and  $\mu_K$ is the usual surface area measure on $\partial K$.  We put
\begin{equation}\label{densities}
p_K(x)=  \frac{\kappa_{K}(x)}{\langle x, N_K(x)\rangle^{n}} \, , \   \ q_K(x)=  \langle x, N_{K}(x) \rangle
\end{equation}
\noindent 
and 
\begin{equation}\label{PQ}
P_K=p_K\  \mu_K \ \ \ \text{and}   \ \ \   Q_K=q_K \ \mu_K.
\end{equation}
Then $P_K$ and $Q_K$  are   measures on $\partial K$ that are absolutely continuous with respect  to $\mu_{K}$. 
Note that 
\begin{equation} \label{volumeK}
\int _{\partial K}  q_K d\mu = n |K| \hskip 5mm \text{and} \hskip 5mm \int _{\partial K}  p_K d\mu = n |K^\circ |.
\end{equation}
The latter holds, provided $K$ has sufficiently smooth boundary.
Thus $Q_K$ and $P_K$  are (up to the factor $n$) the {\em cone measures} (e.g.,  \cite{PaourisWerner2011})  of $K$ and its polar $K^\circ$. 
\par
\noindent
Note that $ \int _{\partial K}  p_K d\mu = n |K^\circ |$ and  $ \int _{\partial K}  q_K d\mu = n |K|.$ 
\vskip 2mm
\par
\noindent
Let $f: (0, \infty) \rightarrow \mathbb{R}$ be a convex or concave function.
 The 
$f$-divergence of $K$ with respect to the measures $P_K$ and $Q_K$ was defined in \cite{Werner2012} as 
\begin{eqnarray}\label{f-div1,0}
D_f(P_K, Q_K)
= \int _{\partial K} f\left( \frac{  \kappa_K(x)} {\langle x, N_K(x) \rangle ^{n+1}}\right)  \langle x, N_K(x) \rangle 
d\mu_K.
\end{eqnarray}
It is a natural generalization of $L_p$-affine surface area and measures the difference between the cone measures of $K$ and $K^\circ$.
Those are relevant in many contexts, 
e.g.,  \cite{HorrmannProchnoThale, Naor} as well as e.g., 
the famous Blaschke Santal\'o inequality and its still open converse, the Mahler conjecture. 
\vskip 2mm
\noindent
Now we turn to applications to convex bodies  of the inequalities we have obtained in the previous sections. There are two equivalent approaches. The first one is to apply the density functions (\ref{densities}) to  Theorem \ref{Gilardoni},  Theorem \ref{thmDragomirOriginal} and Theorem \ref{thmDragomir}. Or, we can apply the log-concave function $ \varphi_K= \exp \left( -\frac{\|\cdot\|_K^2}{2} \right) $ to the  inequalities of the previous sections for log-concave functions.  
Here  
  $\| . \|_K$ is the gauge function of $K$,
\begin{eqnarray*}
\|x\|_K = \min\{\lambda \geq 0: \lambda x \in K\} = \max_{y \in K^\circ}\langle x, y \rangle = h_{K^\circ}(x).
\end{eqnarray*}
Differentiating with respect to $\lam$ at $\lam=1$, we get  
\begin{equation}\label{grad=psi}
\langle x, \nabla \psi(x) \rangle = 2 \psi(x). 
\end{equation} 
 It was already observed in \cite{CFGLSW}
that the $L_\lambda$-affine surface area for log concave functions is a generalization of $L_p$-affine surface area for convex bodies. Indeed, it was noted there that  if one applies  the log concave function $ \varphi_K= \exp \left( -\frac{\|\cdot\|_K^2}{2} \right) $ to  Definition (\ref{asp-Logconcave}), then one obtains the $L_p$-affine surface area for convex bodies, 
\begin{eqnarray}
as_{\lambda}(\varphi_K) = \frac{(2\pi)^{n/2}}{n |B^n_2|} as_p (K),
\end{eqnarray}
where $\lambda = \frac{p}{n+p}$, $p \ne -n$ and  $B^n_2$ denotes the $n$-dimensional Euclidean unit ball. Please note also that 
\begin{equation}\label{integralphi}
 \int e^{ -\frac{\| x \|_K^2}{2} } \ dx = \frac{(2 \pi)^{ \frac{n}{2}} |K|}{|B^n_2|}   \ \ \text{and} \ \  \int e^{ -\frac{\| x \|_{K^\circ}^2}{2} } \ dx = \frac{(2 \pi)^{ \frac{n}{2}} |K^\circ |}{|B^n_2|} .
\end{equation}
\vskip 2mm
\noindent
Now we apply the function,  $ \varphi_K= \exp \left( -\frac{\|\cdot\|_K^2}{2} \right)$, to Corollary \ref{corDragomir} (or apply the densities (\ref{densities}) to Theorem \ref{thmDragomir}) and  obtain the following result. $P_K$, $Q_K$ and  $D_{f}( P_K, Q_K) $ are as above.
\vskip 2mm
\noindent
\begin{theo}\label{thmDragomirBodies}
Let $K$ be a convex body in $\mathbb{R}^n$ with $0$ in its interior. Let  $f:( 0, \infty) \rightarrow  \mathbb{R}$ be  convex and differentiable function, then
\begin{eqnarray}\label{thmDragomirBodiesINEQ}
n |K|  f \left( \frac{|K^\circ |}{|K|}\right)    \ \leq  \ D_f( P_K, Q_K) \
\leq \  n f(1) |K|  +  D_{f'} \left( \frac{ P_{K}^2 }{Q_{K} } , P_{K} \right) - D_{f'}( P_K, Q_K) .
 \end{eqnarray}
\par
\noindent
For a concave, differentiable $f$, the inequalities are reversed. 
\par
\noindent
Equality holds on the left hand side if $K$ is an ellipsoid.  If K is  $C^{2}_+$, then equality holds on the left hand side iff $K$ is an ellipsoid. 
\par
\noindent
Equality  holds  on the right hand side if  $K$ is an origin symmetric ellipsoid
such that $|K| = |B^n_2|$.  If K is $C^{2}_+$, then equality  holds  on the right hand side iff  $K$ is an origin symmetric ellipsoid
such that $|K| = |B^n_2|$.
\end{theo}
\begin{proof}
Let  $\psi=\frac{\|\cdot\|_K^2}{2}$. In \cite{CaglarWerner}, it was proved that 
\begin{eqnarray*}
D_f(P_{\varphi_K}, Q_{\varphi_K}) =  \frac{(2\pi)^\frac{n}{2}}{n|B_2^n|} D_f(P_K, Q_K ).
\end{eqnarray*} 
Clearly,
\begin{eqnarray*}
D_{f'}(P_{\varphi_K}, Q_{\varphi_K}) =  \frac{(2\pi)^\frac{n}{2}}{n|B_2^n|} D_{f'}(P_K, Q_K ).
\end{eqnarray*} 
We integrate in polar coordinates with respect to the cone measure  $Q_K$  (\ref{PQ}) of  $K$. Thus, if we write $x=r z$, with $z \in\partial K$, then $dx=r^{n-1}dr dQ_K(z)$. We also use that  the map  $x\mapsto\det \, {\nabla^2 \psi (x)}$ is $0$-homogeneous. With (\ref{grad=psi}), 
\begin{eqnarray*}
 D_{f'} \left( \frac{ P_{\varphi_K}^2 }{Q_{\varphi_K} } , P_{\varphi_K} \right) &=& \int_0^{+\infty}r^{n-1}e^\frac{-r^2}{2}dr \int_{\partial K} f' \left(\det \, {\nabla^2 \, \psi (z)}\right) \  \det \, {\nabla^2 \, \psi (z)}\ dQ_K(z)\\
 &=& \frac{(2\pi)^\frac{n}{2}}{n |B_2^n|}\int_{\partial K}  f' \left(\det \, {\nabla^2 \, \psi (z)}\right) \ \det \, {\nabla^2 \, \psi (z)} \ dQ_K(z).
\end{eqnarray*}
It is well known (see, e.g.,  \cite{CFGLSW} ) that  for all $z\in\partial K$, 
\begin{equation}\label{hesse}
\det  \,( \nabla^2 \psi ) = \frac{ \kappa_K (z) }{\langle z,  N_K(z)\rangle ^{n+1} }.
\end{equation}
 Thus, 
\begin{eqnarray*}
 D_{f'} \left( \frac{ P_{\varphi_K}^2 }{Q_{\varphi_K} } , P_{\varphi_K} \right)&=& \frac{(2\pi)^\frac{n}{2}}{n|B_2^n|}\int_{\partial K}  f' \left(\frac{ \kappa (x) }{ \langle x,N_K(x)\rangle^{n+1} }\right) \frac{ \kappa (x) }{ \langle x,N_K(x)\rangle^{n} } d\mu_K(x)\\
&=& \frac{(2\pi)^\frac{n}{2}}{n|B_2^n|} D_{f'} \left( \frac{ P_{K}^2 }{Q_{K} } , P_{K} \right).
\end{eqnarray*}
Therefore, the statement of the theorem follows.
\par
\noindent
For ellipsoids, $\ D_f( P_K, Q_K) = n |K| f \left( \frac{|K^\circ |}{|K|}\right)  $ and $\ D_{f'} ( P_K, Q_K) = n  |K| f' \left( \frac{|K^\circ |}{|K|}\right)   $ (\cite{Werner2012}).  And note that (see \cite{Werner2012}) 
$$D_{f'} \left( \frac{ P_{K}^2 }{Q_{K} } , P_{K} \right) = \int f'  \left( \frac{p_K}{q_K} \right) p_K \ d \mu_K  = \int f' \left( \frac{|K^\circ |}{|K|}\right)   p_K d\mu_K = f' \left( \frac{|K^\circ |}{|K|}\right)   n | K^\circ| .$$ 
So, equality holds on the left hand side if $K$ is an ellipsoid.  If K is  $C^{2}_+$, then equality holds on the left hand side iff $p_K = c \ q_K$. This  holds iff $K$ is an ellipsoid by a theorem, due to Petty \cite{Petty}, which says that a $C_{+}^2$ convex body $K$  is an ellipsoid iff 
$$ \frac{\kappa_{K}(z)}{\langle z, N_K(z)\rangle^{n+1}} = c , $$
where $c>0$ a constant. 
\par
\noindent
If $K = \E$  an origin symmetric ellipsoid such that $|\E| = |B^n_2|$, then, by the  Blaschke-Santal\'o  inequality, $|K^\circ| = | \E^\circ| = |B^n_2|$. 
Therefore, equality  holds  on the right hand side if  $K$ is an origin symmetric ellipsoid
such that $|K| = |B^n_2|$.  If K is  $C^{2}_+$, then equality  holds  on the right hand side iff   $p_K =  q_K$, which only holds when
$K$ is an origin symmetric ellipsoid
such that $|K| = |B^n_2|$.
\end{proof}
\vskip 2mm
\noindent
Now we will consider special cases.
\par
\noindent
If we let $f(t) = - \ln t$ in inequality (\ref{thmDragomirBodiesINEQ}),  then by definitions (\ref{def:p-affine}) and  (\ref{K-L-divergence}),  we obtain  the following affine isoperimetric  inequalities.   
\begin{cor}
Let $K$ be a convex body in $\mathbb{R}^n$ with $0$ in its interior. Then
\begin{eqnarray*}
 n |K| \ln \left(\frac{|K|}{|K^\circ|} \right) \ \leq D_{KL} \left( Q_{K} || P_{K}\right) \ \leq  as_{-\frac{n}{2}} (K) - as_0 (K).
\end{eqnarray*}
\par
\noindent
Equality holds on the left hand side if $K$ is an ellipsoid.  If K is  $C^{2}_+$, then equality holds on the left hand side iff $K$ is an ellipsoid. 
\par
\noindent
Equality  holds  on the right hand side if  $K$ is an origin symmetric ellipsoid
such that $|K| = |B^n_2|$.  If K is  $C^{2}_+$, then equality  holds  on the right hand side iff  $K$ is an origin symmetric ellipsoid
such that $|K| = |B^n_2|$.
\end{cor}
\vskip 2mm
\noindent
If we let $f(t) = t^{\frac{p}{n+p}}$ in  Theorem \ref{thmDragomirBodies},  then  by  (\ref{def:p-affine}),  we obtain  the following affine isoperimetric  inequalities.   
Recall that $as_{0} (K) = n |K| $, and if $K$ is  $C^2_+$, then $as_{\infty} (K) = n |K^\circ| $.
\vskip 2mm
 \begin{cor}\label{aff-convex2} 
 Let $K$ be a convex body in $\mathbb{R}^n$ with $0$ in its interior.  Let $p\leq 0$, $p \neq-n$. 
Then
\begin{eqnarray}\label{GAII1}
n|K^\circ|^{\frac {p}{n+p}} |K|^{\frac {n}{n+p}}   \   \leq  \    as_p (K)    \ \leq \  n |K|  + \frac {p}{n+p}   \left(   as_p(K)   -  as_{\frac{- n^2}{2n+p}} (K) \right).
\end{eqnarray}
If  $p>0$, then
\begin{eqnarray}\label{GAII2}
n|K^\circ|^{\frac {p}{n+p}} |K|^{\frac {n}{n+p}}   \   \geq  \    as_p (K)   \ \geq \  n |K|  + \frac {p}{n+p}   \left(   as_p(K)   -  as_{\frac{- n^2}{2n+p}} (K) \right).
\end{eqnarray}
\par
\noindent
Equality holds trivially for $p=0$ or $p = \infty$,   if $K$ is  $C^2_+$.
\par
\noindent
Equality holds on the left hand side if $K$ is an ellipsoid.  If K is  $C^{2}_+$, then equality holds on the left hand side iff $K$ is an ellipsoid. 
\par
\noindent
Equality  holds  on the right hand side if  $K$ is an origin symmetric ellipsoid
such that $|K| = |B^n_2|$.  If K is  $C^{2}_+$, then equality  holds  on the right hand side iff  $K$ is an origin symmetric ellipsoid
such that $|K| = |B^n_2|$.
\end{cor}
\vskip 2mm
\noindent
\textbf{Remark.} 
The $L_p$-affine isoperimetric inequalities state that 
for $p \geq 0$
\begin{equation*}\label{pasa1}
 \frac{as_p(K)}{as_p(B^n_2)}\leq \left(\frac{|K|}{|B^n_2
|}\right)^{\frac{n-p}{n+p}}, 
\end{equation*}
and for $-n<p \leq 0$,
\begin{equation*}\label{pasa2}
\frac{as_p(K)}{as_p(B^n_2)}\geq
\left(\frac{|K|}{|B^n_2|}\right)^{\frac{n-p}{n+p}}.  \    
\end{equation*}
Equality holds trivially if $p=0$.
In both cases
equality  holds for $p \ne 0$ if and only if $K$ is an ellipsoid. 
If $p < -n$ and $K$ is $C^2_+$, then
\begin{equation*} \label{pasa3}
c^{\frac{np}{n+p}}\left(\frac{|K|}{|B^n_2
|}\right)^{\frac{n-p}{n+p}} \leq \frac{as_p(K)}{as_p(B^n_2 )}.
\end{equation*}
These inequalities  were proved by Lutwak \cite{Lutwak1996} for $p > 1$ and for all other $p$ by Werner and Ye
\cite{WernerYe2008} .
\par
\noindent
In the case of a  $0$-symmetric convex body $K$, i.e., $K=-K$, the left hand sides of (\ref{GAII1}) and (\ref{GAII2}),  together with the  Blaschke Santal\'o inequality and its equality characterizations,  (respectively the inverse  Santal\'o inequality \cite{Bourgain-Milman, Kuperberg, Nazarov})  imply  these  inequalities and their equality characterizations.
Moreover, in the last inequality, we remove the $C^2_+$ on $K$ of \cite{WernerYe2008}.
\par
\noindent
Another consequence of Corollary \ref{aff-convex2} are the following 
 Blaschke Santal\'o  type inequalities, which were originally proved in \cite{WernerYe2008} by different methods.  
  \begin{cor}
  Let $K$ be a convex body in $\mathbb{R}^n$ with $0$ in its interior.  
\par
If $p\geq 0$, then
$as_p (K) \ as_p (K^\circ) \leq n^2 |K| |K^\circ|$. 
\par
If $p<0$, $p \neq -n$, then  
$
as_p (K) \  as_p (K^\circ) \geq n^2 |K| |K^\circ|$. 
\par
\noindent
Equality holds trivially for $p=0$ or $p = \infty$,   if $K$ is in $C^2_+$.
\par
\noindent
Equality also holds if $K$ is an ellipsoid.  If K is in $C^{2}_+$, then equality holds  iff $K$ is an ellipsoid. 
 \end{cor}
\vskip 2mm
\noindent
\textbf{Remark.}  
Similarly to the Remarks after  Corollary \ref{aff3}, we have the following consequences of  Corollary \ref{aff-convex2} for a $C^2_+$ convex body $K$:
\par
\noindent
(i)  For $ p >0$, 
\begin{equation*}
 as_p (K) \ \leq \ \left(\frac{p}{n+p} \right)  as_{\infty} (K) +  \left( \frac{n}{n+p} \right)
 as_0 (K) . 
\end{equation*}
The duality relation $as_p(K) =  as_{\frac{n^2}{p}} (K^\circ)$ of \cite{WernerYe2008} then yields
\begin{equation*}
 as_p (K^\circ) \ \leq \ \left(  \frac{p}{n+p}  \right)  as_{0} (K) +  \left( \frac{n}{n+p} \right)  as_{\infty} (K) .
\end{equation*}
For $ p <0$, those  inequalities are reversed.
\par
\noindent
(ii) For $ p \in (-\infty , -n)$,
\begin{eqnarray*}\label{differencebodies}
 as_{\infty} (K) -  as_{0} (K)  \  \leq \   as_p(K)   -  as_{\frac{- n^2}{2n+p}} (K).
\end{eqnarray*}
For $ p > -n$, the inequality is reversed.
\par
\noindent
Note also that similar results hold for  the dual body $K^\circ$   by the above duality relation.
\par
\noindent
(iii)  For $p=n$,  Corollary \ref{aff-convex2}  yields
\begin{eqnarray*}
(as_{n} (K))^2  \  \leq \   as_0 (K) \  as_{\infty} (K).
\end{eqnarray*}
\vskip 5mm
\noindent
In \cite{CaglarWerner}, it was proved that  for normalized densities
\begin{eqnarray*}
D_f(P_{\varphi_K}, Q_{\varphi_K}) =   D_f(P_K, Q_K ) 
\end{eqnarray*} 
Now if we apply $ \varphi_K= \exp \left( -\frac{\|\cdot\|_K^2}{2} \right) $ to the inequality (\ref{main2}), we obtain similarly
$$
 D_f(P_K, Q_K ) \geq \frac{f'' (1)}{2} \  \left(  \int   \left|   \frac{ |B^n_2| \ e^{-\frac{\|z\|_K^2}{2} }  }{(2\pi)^{\frac{n}{2}} }  \left( \frac{ \kappa_K (z) }{|K^{\circ}| \langle z,  N_K(z)\rangle ^{n+1} }  -  \frac{1}{|K|} \right)  \right|  dz  \right)^2 .
$$
\noindent
The relation between the normalized cone measure $Q_K$ and the Hausdorff measure $\mu_K$ on $\partial K$ is given by 
$$
dQ_K(x)=\frac{\langle x,N_K(x)\rangle d\mu_K(x)}{n|K|}.
$$
We integrate in polar coordinates with respect to the normalized cone measure  $Q_K$  of  $K$. Thus, if we write $x=r z$, with $z \in\partial K$, then $dx= n |K |r^{n-1}dr dQ_K(z)$. We also use that  the map  $x\mapsto\det \, {\nabla^2 \psi (x)}$ is $0$-homogeneous. So, we obtain, 
\begin{eqnarray*}
 D_f(P_K, Q_K ) \hskip -3mm  & \geq&  \hskip -3mm  \frac{f'' (1)}{2}  \frac{ n^2 |K|^2  |B^n_2|^2 }{(2\pi)^n } \hskip -1mm   \left(  \int_0^{+\infty} \hskip -5mm r^{n-1}e^\frac{-r^2}{2}dr  \hskip -2mm  \int_{\partial K}  \hskip -1mm   \left|   \frac{ \kappa_K (z) }{|K^{\circ}| \langle z,  N_K(z)\rangle ^{n+1} }  -  \frac{1}{|K|}   \right|  dQ_K(z)  \hskip -1mm   \right)^2   \nonumber \\ 
&=&  \frac{f'' (1)}{2 } |K|^2  \left(  \int_{\partial K}   \left|  \frac{ \kappa_K (z) }{|K^{\circ}| \langle z,  N_K(z)\rangle ^{n+1} }  -  \frac{1}{|K|}   \right|    \ dQ_K(z)  \right)^2   \nonumber \\
&=&   \frac{f'' (1)}{2 }   \left(  \int_{\partial K}   \left|  \frac{ \kappa_K (z) }{|K^{\circ}| \langle z,  N_K(z)\rangle ^{n+1} }  -  \frac{1}{|K|}   \right|    \  \frac{ \langle z,  N_K(z)\rangle }{n} d\mu_K(z)   \right)^2  \nonumber \\
&=&  \frac{f'' (1)}{2 }   \left(  \int_{\partial K}   \left|  \frac{ \kappa_K (z) }{ n |K^{\circ}| \langle z,  N_K(z)\rangle ^{n} }  -  \frac{ \langle z,  N_K(z)\rangle }{n |K|}   \right|    d\mu_K(z)   \right)^2 \nonumber \\  &= & \frac{f'' (1)}{2 }  \ V^2 (P_K , Q_K).
\end{eqnarray*}
\vskip 2mm
\noindent
Thus we obtained the following Pinsker type inequality for convex bodies.
 \begin{cor}
Let  $K$ be a convex body in $\mathbb{R}^n$ with $0$ in its interior.  And let $f: (0, \infty) \rightarrow \R$ be convex  and $f(1) =0$. 
Suppose that the convex function $f$ is differentiable up to order 3 at 1 with $f''(1) >0$ and the following inequality holds
\begin{eqnarray*}
\bigg( f(u) - f'(1) (u-1) \bigg) \left( 1 - \frac{f'''(1)}{3 f''(1)}(u-1) \right) \ \geq \ \frac{f''(1)}{2} (u-1)^2 .
\end{eqnarray*}
Then
\begin{eqnarray*}
 D_f(P_K, Q_K )  \geq  \frac{f'' (1)}{2 }  \ V^2 (P_K , Q_K).
\end{eqnarray*}
\par
\noindent
If $f$ is concave, the inequalities are reversed. Equality also holds if $K$ is an ellipsoid.
 \end{cor}
\par
\noindent 
Similarly, Corollary \ref{PinskerIneqFunctional}  (the case $f(t) = - \ln t $) becomes
$$
D_{KL} (Q_K || P_K) \geq  \frac{1}{2}  \ V^2 (P_K , Q_K).
$$

\vskip 2mm 
\noindent 
Umut Caglar\\
{\small Department of Mathematics and Statistics} \\
{\small Florida International University} \\
{\small Miami, FL 33199, U. S. A. }   \\
{\small \tt ucaglar@fiu.edu}\\ \\
\vskip 3mm
\noindent
Alexander V. Kolesnikov  \\
{\small Faculty of mathematics} \\
{\small National Research Institute Higher School of Economics} \\
{\small Moscow, Russia}   \\
{\small \tt sascha77@mail.ru}\\ \\
\vskip 3mm

\vskip 3mm
\noindent
Elisabeth Werner\\
{\small Department of Mathematics \ \ \ \ \ \ \ \ \ \ \ \ \ \ \ \ \ \ \ Universit\'{e} de Lille 1}\\
{\small Case Western Reserve University \ \ \ \ \ \ \ \ \ \ \ \ \ UFR de Math\'{e}matique }\\
{\small Cleveland, Ohio 44106, U. S. A. \ \ \ \ \ \ \ \ \ \ \ \ \ \ \ 59655 Villeneuve d'Ascq, France}\\
{\small \tt elisabeth.werner@case.edu}\\ \\


\begin{thebibliography}{10}

\bibitem{Alexandroff}
{\sc A.D. Alexandroff}, 
{\em  Almost everywhere existence of the second differential of a convex
 function and some properties of convex surfaces connected with it},
(Russian)  Leningrad State Univ. Annals [Uchenye Zapiski] Math. Ser.  6,  (1939),  3--35.

\vskip 2mm


\bibitem{Alonso-Gutierrez}
{\sc D. Alonso-Guti\'errez}, {\em A Reverse Rogers–Shephard Inequality for Log-Concave Functions},
Journal of Geometric Analysis 29, (2019), 299--315.

\vskip 2mm


\bibitem{Alonso-Gutierrez2016}
{\sc D. Alonso-Guti\'errez,  B. G. Merino, C. H. Jim\'enez,  and R. Villa}, {\em Rogers-Shephard inequality for log-concave functions},
Journal of Functional Analysis 271(11), (2016), 3269--3299.

\vskip 2mm


\bibitem{Alonso-Gutierrez2017}
{\sc D. Alonso-Guti\'errez,  B. G. Merino, C. H. Jim\'enez,  and R. Villa}, 
{\em John's Ellipsoid and the Integral Ratio of a Log-Concave Function}.
Journal of Geometric Analysis,  28: (2018), 1182--1201. 


\vskip 2mm

\bibitem {AliSilvery1966}
{\sc M. S. Ali and  D. Silvey}, {\em A general class of coefficients of divergence
of one distribution from another}, Journal of the Royal Statistical Society, Series B 28, (1966), 131--142.

\vskip 2mm

\bibitem {ArtKlarMil}
{\sc S. Artstein-Avidan,  B. Klartag and  V. Milman}, {\em The
Santal\'o point of a function, and a functional form of Santal\'o
inequality}, Mathematika 51, (2004), 33--48.

\vskip 2mm

\bibitem{ArtKlarSchuWer}
{\sc S. Artstein-Avidan,  B. Klartag, C. Sch\"utt and E. Werner}, {\em 
Functional affine-isoperimetry and an inverse logarithmic Sobolev inequality},
Journal of Functional Analysis, vol. 262, no.9, (2012), 4181--4204.

\vskip 2mm



\bibitem{Aubrun-SzarekBook}
{\sc G. Aubrun and S. Szarek}, {\em
Alice and Bob Meet Banach. The Interface of Asymptotic Geometric Analysis and Quantum Information Theory}, 
Mathematical Surveys and Monographs, Vol. 223, Amer. Math. Soc. 2017. 

\vskip 2mm

\bibitem{Aubrun-Szarek-Werner2010}
 {\sc G. Aubrun, S. Szarek, and E. Werner}, {\em Non-additivity of Rényi entropy and Dvoretzky's theorem},   
 J. Math. Phys. 51, 022102 (2010).

\vskip 2mm

\bibitem{Aubrun-Szarek-Werner2011}
 {\sc G. Aubrun, S. Szarek, and E. Werner}, {\em Hastings's additivity counterexample via Dvoretzky's theorem},   
Comm. Math. Physics 305, (2011), 85--97.

\vskip 2mm

\bibitem{Aubrun-Szarek-Ye2012} 
{\sc G. Aubrun, S. Szarek, and D. Ye}, {\em Phase transitions for random states and a semicircle law for the partial transpose},   
Phys. Rev. A. 85, 030302(R) (2012). 

\vskip 2mm

\bibitem{Aubrun-Szarek-Ye2014}
 {\sc G. Aubrun, S. Szarek, and D. Ye}, {\em Entanglement thresholds for random induced states},   
Comm. Pure Appl. Math. 67 (2014), 129--171.


 \vskip 2mm


\bibitem{KBallthesis}
{\sc K. Ball}, {\em Isometric problems in $ l_p $ and sections of convex sets}, PhD dissertation, University
of Cambridge (1986).


 
 \bibitem{BaranyLarman1988}
{\sc I.~B\'ar\'any and D.G.~Larman},
{\em Convex bodies, economic cap coverings, random polytopes},
Mathematika 35, (1988), 274--291.

 
 
 \vskip 2mm

\bibitem{BarronGyorfiMeulen}
{\sc A. R. Barron, L. Gy\"orfi and E.C. van der Meulen}, {\em Distribution estimates consistent in total variation and two types of information divergence},  IEEE Trans. Inform. Theory 38, (1990), 1437--1454.

\vskip 2mm
 
 
\bibitem {Barron1986}
{\sc A. R. Barron}, {\em Entropy and the central limit theorem}, Ann. Probab., vol. 14, no. 1, (1986),
336--342. 

\vskip 2mm

\bibitem{Besau-Werner2016}
{\sc  F. Besau and  E. Werner}, {\em The Spherical Convex Floating Body}, 
Advances in Mathematics 301,  (2016), 867--901.

\vskip 2mm

\bibitem{Besau-Werner2018}
{\sc  F. Besau and  E. Werner}, {\em The Floating Body in Real Space Forms}, 
Journal of Differential Geometry, Vol. 110, No. 2,  (2018), 187--220.


\vskip 2mm


\bibitem{Blaschke}
{\sc W. Blaschke}, {\em {\"U}ber affine Geometrie VII. Neue Extremeigenschaften von
Ellipse und Ellipsoid}, Leipz. Ber. 69, (1917), 306--318.

\vskip 2mm

\bibitem{Boeroetzky2000}
{\sc K. Jr. B\"or\"oczky}, {\em Polytopal approximation bounding the number of $k$-faces},
 Journal of Approximation Theory 102, (2000), 263--285.

\vskip 2mm

\bibitem {Boeroetzky2000/2} 
{\sc K. Jr. B\"or\"oczky}, {\em Approximation of general smooth convex bodies},  Advances in Mathematics  153, (2000), 325--341.

\vskip 2mm

\bibitem{BoeroetzkyReitzner2004} 
{\sc K. Jr. B\"or\"oczky and M. Reitzner}, {\em Approximation of smooth convex bodies by random circumscribed polytopes},  
Ann. Appl. Probab.  14, (2004), 239--273.

\vskip 2mm

\bibitem{Rademacher}
{\sc J. M. Borwein and J.D. Vanderwerff}, 
{\em Convex Functions: Constructions, Characterizations and Counterexamples}, 
Cambridge University Press 2010.

\vskip 2mm

 \bibitem {Bolley-Villani}
 {\sc F. Bolley and C. Villani}, {\em Weighted Csiszár-Kullback-Pinsker inequalities and applications to transportation inequalities}, Annales de la Faculté des sciences de Toulouse : Mathématiques, Serie 6, Volume 14, no. 3 (2005),  331--352.
 

\vskip 2mm

 \bibitem {Bourgain-Milman}
 {\sc J. Bourgain and V.D. Milman}, {\em New volume ratio properties for convex symmetric bodies in $\R^n$}, Invent. Math. 88 (2)
(1987), 319-–340.




\vskip 2mm

 \bibitem {Brunel}
 {\sc V. Brunel}, {\em Concentration of the empirical level sets of Tukey’s halfspace depth},
Probab. Theory Relat. Fields 173, (2019),1165--1196.



\vskip 2mm

\bibitem{Buse-Feller}
{\sc H. Busemann and W. Feller}, 
{\em Kr\"ummungseigenschaften konvexer Fl\"achen}, 
Acta Math.   66, (1935), 1--47.

\vskip 2mm


\bibitem{CaglarWerner}
{\sc  U.\ Caglar and E.\ Werner},
 {\em Divergence for $s$-concave and log concave functions}, Advances in Mathematics  257, (2014),   219--247.

\vskip 2mm

\bibitem{CaglarWerner2}
{\sc  U.\ Caglar and E.\ Werner}, {\em Mixed $f$-divergence and inequalities for log concave functions}, Proc. London Math. Soc., vol. 210, (2015), 271--290.

\vskip 2mm

\bibitem{CFGLSW} 
{\sc U.\ Caglar, M.\ Fradelizi, O.\ Guedon, J.\ Lehec, C.\ Sch\"utt and E.\ Werner}, {\em Functional versions of $L_p$-affine surface area and entropy inequalities}, Int. Math. Res. Not., vol. 2016, (2016), 1223--1250.

\vskip 2mm

\bibitem{CaglarYe}
{\sc  U.\ Caglar and D.\ Ye}, {\em Affine isoperimetric inequalities in the functional Orlicz–Brunn–Minkowski theory}, 
Advances in Applied Mathematics 81, (2016), 78--114.

\vskip 2mm


\bibitem{Colesanti}
{\sc A. Colesanti}, {\em 
Functional inequalities related to the Rogers-Shephard inequality},   Mathematica, vol. 53,  (2006),  81--101.

\vskip 2mm

\bibitem{Colesanti-Fragala}
{\sc A. Colesanti and I. Fragal\'a}, {\em 
 The first variation of the total mass of log-concave functions and related inequalities},  Advances in Mathematics 244, (2013), 708--749.

\vskip 2mm

\bibitem{Colesanti-Mussnig-Ludwig-1}
{\sc A. Colesanti,  F. Mussnig and M. Ludwig}, {\em 
A homogeneous decomposition theorem for valuations on convex functions},  To appear in Journal of Functional Analysis.

\vskip 2mm


\bibitem{CoverThomas2006}
{\sc T. Cover and J. Thomas}, {\em Elements of information theory}, second ed., Wiley-Interscience,  (John Wiley and Sons), Hoboken, NJ, (2006). 

\vskip 2mm


\bibitem{Csiszar}
{\sc I. Csisz\'ar}, {\em Eine informationstheoretische Ungleichung und ihre Anwendung
auf den Beweis der Ergodizit{\"a}t von Markoffschen Ketten}, Publ.
Math. Inst. Hungar. Acad. Sci. ser. A 8, (1963), 84--108.

\vskip 2mm

\bibitem {Csiszar1967}
{\sc I. Csisz\'ar}, {\em Information-type measures of difference of probability distributions and indirect
observations}, Studia Sci. Math. Hungar., vol. 2, (1967),  299--318. 


\vskip 2mm

\bibitem {Csiszar1984}
{\sc I. Csisz\'ar}, {\em Sanov property, generalized I-projection and a conditional limit theorem},  Ann. Probab., vol. 12, no. 3,  (1984),  768--793.

\vskip 2mm

\bibitem{DCT}
{\sc A. Dembo, T. Cover and J. Thomas}, {\em Information theoretic inequalities}, IEEE Trans. Inform. Theory 37, (1991), 1501--1518.

\vskip 2mm

\bibitem {DRAGOMIR_1}
{\sc S.S. Dragomir}, {\em Inequalities for Csisz\'ar $f$-divergence in information theory}, RGMIA Monographs, Victoria University,  (2000).


\vskip 2mm

\bibitem {Fedotov-Harremoes-Topsoe2003/1}
{\sc A. Fedotov, P. Harremo{\"e}s and F. Tops{\o}e}, {\em Best Pinsker bound equals Taylor polynomial of degree 49},
Computational Technologies, vol. 8, (2003), 3--14.

\vskip 2mm

\bibitem {Fedotov-Harremoes-Topsoe2003/2}
{\sc A. Fedotov, P. Harremo{\"e}s and F. Tops{\o}e}, {\em Refinements of Pinsker's inequality},
IEEE Trans. Inf. Theory, vol. 49, no. 6, (2003), 1491--1498.

\vskip 2mm


\bibitem {Figalli}
{\sc A. Figalli}, {\em The Monge-Amp\`ere equation and its applications},
Zurich Lectures in Advanced Mathematics.
European Mathematical Society (EMS), Zurich, (2017). ISBN 978-3-03719-170-5.
\vskip 2mm



\bibitem{FradeliziMeyer2007}
{\sc M. Fradelizi and  M. Meyer}, {\em  Some functional forms of
Blaschke-Santal\'o inequality}, Math. Z. 256, no. 2, (2007),
379--395.

\vskip 2mm


\bibitem{FM2008}
{\sc M. Fradelizi and  M. Meyer}, {\em  Increasing functions and inverse Santal\'o inequality for unconditional functions}, Positivity 12, no. 3, (2008),  407--420.

\vskip 2mm

\bibitem{GardnerBook}
{\sc R. J. Gardner}, {\em Geometric tomography}, 
Cambridge University Press (1995).

\vskip 2mm

\bibitem{Gardner2002}
{\sc R. J. Gardner}, {\em The Brunn-Minkowski inequality},  Bull. Amer. Math. Soc. 39, (2002), 355--405.

\vskip 2mm

\bibitem {Gilardoni2010}
{\sc G. L. Gilardoni}, {\em On Pinsker's and Vajda's type inequalities for Csisz\'ar's $f$-divergences},
IEEE Trans. Inf. Theory, vol.56, no. 11, (2010), 5377--5386.



\vskip 2mm

\bibitem {GTW2018}
{\sc J. Grote,  C. Thäle and E. Werner}, {\em
Surface area deviation between smooth convex bodies and polytopes},   arXiv:1811.04656, (2018). 

\vskip 2mm

\bibitem {GroteWerner2018}
{\sc J. Grote and E. Werner}, {\em
Approximation of smooth convex bodies by random polytopes}, 
Electronic Journal of Probability 23,  no 9, (2018). 


\vskip 2mm



\bibitem{GLYZ}
{\sc O.G. Guleryuz, E. Lutwak, D. Yang, and G. Zhang}, {\em Information theoretic inequalities for
contoured probability distributions}, IEEE Trans.  Inf.  Theory, 48, (2002), 2377--2383.

\vskip 2mm


\bibitem{HabSch2}
{\sc C. Haberl and F.  Schuster},  {\em General Lp-affine isoperimetric inequalities},
J. Differential Geometry  83, (2009), 1--26.


\vskip 2mm
\bibitem{HanSlomkaWerner}
{\sc H.~Huang, B.~Slomka, and E.~Werner}, {\em Ulam floating bodies}, 
Journal of London Math. Society 100, (2019),  425--446.

\vskip 2mm

\bibitem{HarremoesTopsoe}
{\sc P. Harremoes and F. Tops{\o}e}, {\em Inequalities between entropy and the index of coincidence derived from information diagrams}, IEEE Trans. Inform. Theory 47,  (2001),  2944--2960.

\vskip 2mm

\bibitem{HorrmannProchnoThale}
{\sc J. Hörrmann,  J. Prochno, and C. Thäle}, {\em Isotropic constant of random polytopes with vertices on an $\ell_p$-sphere }, The Journal of Geometric Analysis 28, (2018), 405--426.



\vskip 2mm

\bibitem{JenkinsonWerner}
{\sc J. Jenkinson and E. Werner}, {\em Relative entropies for convex bodies},  Trans. Amer. Math. Soc.  366, (2014), 2889--2906.


\vskip 2mm

\bibitem{KK} 
{\sc B. Klartag  and A.V. Kolesnikov},
{\em  Eigenvalue distribution of optimal transportation},
 Analysis \& PDE, Vol. 8, No. 1, (2015),  33--55.

\vskip 2mm


\bibitem {Kemperman1969}
{\sc J. H. B. Kemperman}, {\em On the optimal rate of transmitting information},
Ann. Math. Statist, vol. 40, (1969), 2156--2177. 


\vskip 2mm

\bibitem {KoldobskyZvavitch}
{\sc A. Koldobsky and A. Zvavitch}, {\em 
An isomorphic version of the Busemann-Petty problem for arbitrary measures},  Geometriae Dedicata, vol. 174, Issue 1 (2015), 261--277.

\bibitem{Kol2010} 
{\sc A.V. Kolesnikov}, {\em On Sobolev regularity of mass transport and transportation inequalities},
Theory Probab. Appl., vol. 57(2), (2013), 24--264.


\vskip 2mm


\bibitem{Kol} 
{\sc A.V.~Kolesnikov}, 
{\em  Hessian metrics, CD(K,N)-spaces, and optimal transportation of log-concave measures},
Discrete and Continuous Dynamical Systems,  Series A  vol. 34. no. 4,  (2014), 1511--1532.

\vskip 2mm


\bibitem {Kullback1967}
{\sc S. Kullback}, {\em A lower bound for discrimination information in terms of variation},
IEEE Trans. Inf. Theory, vol. IT-13, (1967), 126--127.

\vskip 2mm

\bibitem {Kullback1970}
{\sc S. Kullback}, {\em Correction to: A lower bound for discrimination information in terms of variation},
IEEE Trans. Inf. Theory, vol. IT-16, (1970), 652--652.

\vskip 2mm

 \bibitem{KullbackLeibler1951}
{\sc S. Kullback and R. Leibler,} {\em On information and sufficiency},  Ann.
 Math. Statist., 22 (1951), 79--86.
 


\vskip 2mm

\bibitem{Kuperberg}
{\sc G. Kuperberg}, {\em From the Mahler conjecture to Gauss linking
integrals}, Geometric and Functional Analysis  18, (2008), 870--892.


\vskip 2mm

\bibitem{GaZ}
{\sc R. J. Gardner and G. Zhang},
{\em Affine inequalities and radial mean bodies}, 
 Amer. J. Math. 120, no.3, (1998), 505--528.

\vskip 2mm

\bibitem{HabSch2}
{\sc C. Haberl and F.  Schuster,} {\em General Lp affine isoperimetric inequalities},
J. Differential Geometry  83, (2009), 1--26.

\vskip 2mm


\bibitem{Lehec2009}
{\sc J. Lehec}, {\em A simple proof of the functional Santal\'o inequality}, C. R. Acad. Sci. Paris. S\'er.I
347, (2009), 55--58.

\vskip 2mm



\bibitem{Leichtweiss:1986a}
{\sc K.~Leichtweiss},
{\em Zur {A}ffinoberfl{\"a}che konvexer {K}{\"o}rper ({G}erman)},
Manu\-scripta Math. 56 (1986), 429--464.


\vskip 2mm



\bibitem{LSW2019}
{\sc B. Li, C. Sch{\"u}tt and E. Werner}, 
{\em The L\"owner function of a log-concave function},  
arXiv: 1904.01211, to appear in Journal of Geometric Analysis.

\vskip 2mm


\bibitem{LieseVajda1987}
{\sc F. Liese and I. Vajda},
{\em Convex Statistical Distances},
Leipzig, Germany:Teubner, (1987).

\vskip 2mm

\bibitem{LieseVajda2006}
{\sc F. Liese and I. Vajda},
{\em On Divergences and Information in Statistics and Information Theory}, IEEE Trans. Inf. Theory  52,  (2006), 4394--4412.

\vskip 2mm


\bibitem{Ludwig-Reitzner1999}
{\sc M. Ludwig and M. Reitzner}, {\em A characterization of affine surface area},  Advances in Mathematics 147, (1999), 138--172.

\vskip 2mm



\bibitem{Ludwig-Reitzner}
{\sc M. Ludwig and M. Reitzner}, {\em A classification of $SL(n)$
invariant valuations},  Annals of Math. 172, (2010), 1223--1271.

\vskip 2mm

\bibitem{Lutwak1991}
{\sc E.~Lutwak}, 
{\em Extended affine surface area}, 
Advances in Mathematics  85 (1991), 39--68.

\vskip 2mm

\bibitem{LutwakOliker1995}
{\sc E. Lutwak and V. Oliker}, 
{\em On the regularity of solutions to a generalization of the Minkowski problem},
J. Differential Geometry  41 (1995), 227--246.



\vskip2mm

\bibitem{Lutwak1996}
{\sc E. Lutwak}, {\em The Brunn-Minkowski-Firey theory II : Affine and
geominimal surface areas}, Advances in Mathematics 118, (1996), 244--294.



\vskip 2mm

\bibitem{LutwakYangZhang2002/1} 
{\sc E. Lutwak, D. Yang and G. Zhang}, {\em The Cramer--Rao inequality for star bodies},
Duke Math. J. 112, (2002), 59--81.


\vskip 2mm

\bibitem{LutwakYangZhang2004/1} 
{\sc E. Lutwak, D. Yang and G. Zhang}, {\em Moment-entropy inequalities},
Ann. Probab. 32, (2004), 757--774. 

\vskip 2mm


\bibitem{LutwakYangZhang2005} 
{\sc E. Lutwak, D. Yang and G. Zhang}, {\em Cramer-Rao and moment-entropy inequalities for R\'enyi entropy and generalized Fisher information},
IEEE Trans. Inf. Theory  51, (2005), 473--478.
\vskip 2mm


\bibitem{MW2}
{\sc M. Meyer and E. Werner}, {\em On the p-affine surface area},
Advances in Mathematics 152, (2000), 288--313.

\vskip 2mm


\bibitem{Morimoto1963}  
{\sc T. Morimoto}, {\em Markov processes and the H-theorem}, 
J. Phys. Soc. Jap. 18, (1963), 328--331.



\vskip 2mm

 \bibitem{NagySchuettWerner}
  {\sc S. Nagy, C. Sch\"utt and E.  Werner},
{\em  Data depth and floating body}  Statistics  Surveys  13, No. 0  (2019),  52--118.


\vskip 2mm


\bibitem{Naor}
{\sc A. Naor}, {\em The surface measure and cone measure on the sphere of $\ell_p^n$},
Trans. Amer. Math. Soc. 359 (2007), 104--1079.


\vskip 2mm

\bibitem{Nazarov}
{\sc F. Nazarov}, {\em The H\"ormander proof of the Bourgain-Milman theorem},
Geometric Aspects of Functional Analysis, 
Lecture Notes in Mathematics, vol. 2050, (2012), 335--343.

\vskip 2mm


\bibitem{OsterrVajda}
{\sc F.  {\"O}sterreicher and I. Vajda},
{\em A new class of metric divergences on probability spaces and its applicability in statistics},
Ann. Inst. Statist. Math.,  55,  (2003), 639--653.

\vskip 2mm

\bibitem{PaourisWerner2011} 
{\sc G. Paouris and E.  Werner}, {\em Relative entropy of cone measures and $L_p$ centroid bodies},   
Proceedings London Math. Soc. (3) 104,  (2012),  253--286. 

\vskip 2mm



\bibitem{Petty}
 {\sc C. Petty}, {\em Affine isoperimetric problems},   
Discrete geometry and convexity, Annals of the New York Academy of Sciences 440 (Wiley-Blackwell, New York, 1985) 113–-127. 

\vskip 2mm

\bibitem {Pinsker1960}
{\sc M. S. Pinsker}, {\em Information and Information Stability of Random Variables and Processes},
Holden-Day, SanFrancisco, CA. 1960 (English ed., 1964, translated and edited by Amiel Feinstein).

\vskip 2mm 

\bibitem {Reid-Williamson2009}
{\sc M. D. Reid and R. C. Williamson}, {\em Generalized Pinsker  Inequalities},	 CoRR abs/0906.1244, (2009).

\vskip 2mm


\bibitem{Reitzner}
{\sc M. Reitzner}, {\em Random points on the boundary of smooth convex bodies}, 
Trans. Amer. Math. Soc. 354,  (2002), 2243--2278.



\vskip 2mm


\bibitem{Ren}
{\sc A. R\'enyi}, {\em On measures of entropy and information}, Proceedings  of the 4th Berkeley Symposium on  Probability Theory and Mathematical  Statistics, vol.1 (1961), 547-561.

\vskip 2mm 

\bibitem{Rockafellar}
{\sc R.T. Rockafellar},
{\em  Convex analysis.}
Reprint of the 1970 original.
Princeton Landmarks in Mathematics. Princeton Paperbacks.
Princeton University Press, Princeton, NJ,  (1997). xviii+451 pp. ISBN: 0-691-01586-4.



\vskip 2mm


\bibitem{Rotem}
{\sc L. Rotem}, {\em On the Mean Width of Log-Concave Functions},
In: Klartag B., Mendelson S., Milman V. (eds) Geometric Aspects of Functional Analysis. Lecture Notes in Mathematics, vol 2050 (2012). Springer, Berlin, Heidelberg.

\vskip 2mm


\bibitem{Santalo} 
{\sc L.A. {S}antal\ensuremath{\mbox{\'o}}}, {\em An affine invariant for convex bodies of n-dimensional space}, 
(Spanish) Portugaliae Math. 8, (1949), 155--161.


\vskip 2mm


 \bibitem{SchneiderBook}
  {\sc R. Schneider,}  {\em Convex Bodies: The Brunn-Minkowski
 theory,} Cambridge Univ. Press, 1993.


\vskip 2mm

\bibitem{Schuster2010}
{\sc F. Schuster},  {\em Crofton measures and Minkowski valuations}, Duke Math. J.
154, (2010), 1--30.


\bibitem{SchusterWannerer}
{\sc F. Schuster and M. Weberndorfer},  {\em GL(n) contravariant Minkowski valuations}
Trans. Amer. Math. Soc. 364 (2012), no. 2, 815--826.

\vskip 2mm

\bibitem{Szarek-Werner-Zyczkowski}
 {\sc S. Szarek, E. Werner, and  K. Zyczkowski }, {\em How often is a random quantum state k-entangled?},   
 J. Phys. A: Math. Theor. 44, 045303 (2011).
 

\vskip 2mm

\bibitem{SW1990} 
{\sc C. Sch\"utt and E. Werner}, {\em The convex floating body},   
Math. Scand.  66, (1990), 275--290.

\vskip 2mm

\bibitem{SchuettWerner2003}
{\sc C. Sch\"utt  and E. Werner}, {\em Polytopes with vertices chosen randomly from the boundary of a convex body}, 
Geometric aspects of functional analysis, Lecture Notes in Math. 1807. Springer-Verlag, (2003), 241--422 .


\vskip 2mm

\bibitem{SW2004}
{\sc C. Sch{\"u}tt and E. Werner}, {\em Surface bodies and
$p$-affine surface area},  Advances in Mathematics 187, (2004), 98--145.


\vskip 2mm

\bibitem{Stancu}
{\sc A. Stancu},  {\em The Discrete Planar $L_0$-Minkowski
Problem},  Advances in Mathematics 167,  (2002),  160--174.

\vskip 2mm

\bibitem {Topsoe1979}
{\sc F. Tops{\o}e}, {\em Information theoretical optimization techniques},
Kybernetika, vol. 15, no. 1, (1979),  8--27.

\vskip 2mm


\bibitem{Trudinger-Wang}
 {\sc N. S. Trudinger and X. Wang}, {\em The affine Plateau problem},
  J. Amer. Math. Soc., vol. 18, (2005), 253--289.

\vskip 2mm


\bibitem{Villani} 
{\sc C. Villani}, {\em  Topics in optimal transportation},
Amer. Math. Soc. Providence, Rhode Island, 2003.


\vskip 2mm


\bibitem{Wang} 
{\sc X. Wang}, {\em  Affine maximal hypersurfaces},  Proceedings of the International Congress
of Mathematicians, vol. III, (2002), Beijing, 221--231.

\vskip 2mm

\bibitem{Werner2012/1}
{\sc E. Werner}, {\em
 R\'enyi Divergence and $L_p$-affine surface area  for convex bodies},
Advances in Mathematics 230, (2012), 1040--1059.

\vskip 2mm

\bibitem{Werner2012}
 {\sc E. Werner}, {\em f-Divergence for convex bodies}, 
Proceedings of  the ``Asymptotic Geometric Analysis" workshop, Fields Institute, Toronto,  (2012).

\vskip 2mm

\bibitem{WernerYe2008}
 {\sc E. Werner and D. Ye}, {\em New $L_p$-affine isoperimetric inequalities},   
Advances in Mathematics  218, (2008), 762--780. 



\end{thebibliography}
\end{document}